\documentclass[11pt,a4paper,reqno]{amsart}
\usepackage[english]{babel}
\usepackage[applemac]{inputenc}
\usepackage[T1]{fontenc}
\usepackage{palatino}
\usepackage{amsmath}
\usepackage{amssymb}
\usepackage{amsthm}
\usepackage{amsfonts}
\usepackage{graphicx}
\usepackage{color}
\usepackage{vmargin}
\usepackage{mathtools}
\usepackage{comment}

\setmarginsrb{20mm}{20mm}{20mm}{20mm}{10mm}{10mm}{10mm}{10mm}

\theoremstyle{plain}
\newtheorem{theorem}{Theorem}[section]
\newtheorem{prop}[theorem]{Proposition}
\newtheorem{lemma}[theorem]{Lemma}
\newtheorem{corol}[theorem]{Corollary}
\theoremstyle{definition}
\newtheorem{definition}{Definition}[section]

\newtheorem{remark}{\textnormal{\textbf{Remark}}}[section]
\newtheorem*{remark*}{\textnormal{\textbf{Remark}}}

\def\vint{\mathop{\mathchoice%
          {\setbox0\hbox{$\displaystyle\intop$}\kern 0.22\wd0%
           \vcenter{\hrule width 0.6\wd0}\kern -0.82\wd0}%
          {\setbox0\hbox{$\textstyle\intop$}\kern 0.2\wd0%
           \vcenter{\hrule width 0.6\wd0}\kern -0.8\wd0}%
          {\setbox0\hbox{$\scriptstyle\intop$}\kern 0.2\wd0%
           \vcenter{\hrule width 0.6\wd0}\kern -0.8\wd0}%
          {\setbox0\hbox{$\scriptscriptstyle\intop$}\kern 0.2\wd0%
           \vcenter{\hrule width 0.6\wd0}\kern -0.8\wd0}}%
          \mathopen{}\int}
\def\vintslides{\mathop{\mathchoice%
          {\setbox0\hbox{$\displaystyle\intop$}\kern 0.22\wd0%
           \vcenter{\hrule height 0.04em width 0.6\wd0}\kern -0.82\wd0}%
          {\setbox0\hbox{$\textstyle\intop$}\kern 0.2\wd0%
           \vcenter{\hrule height 0.04em width 0.6\wd0}\kern -0.8\wd0}%
          {\setbox0\hbox{$\scriptstyle\intop$}\kern 0.2\wd0%
           \vcenter{\hrule height 0.04em width 0.6\wd0}\kern -0.8\wd0}%
          {\setbox0\hbox{$\scriptscriptstyle\intop$}\kern 0.2\wd0%
           \vcenter{\hrule height 0.04em width 0.6\wd0}\kern -0.8\wd0}}%
          \mathopen{}\int}

\def \harm{\mathcal{H}}

\def \codim{{\rm{codim}}}
\def \div {{\rm div}}

\newcommand*{\abs}[1]{\left\vert{#1}\right\vert}

\newcommand*{\Lip}{\mathrm{Lip}\,}

\newcommand*{\dist}{\mathrm{dist}}

\newcommand{\spt}{\operatorname{spt}}

\newcommand{\N}{{\mathbb N}}
\newcommand{\R}{{\mathbb R}}

\newcommand{\ep}{\epsilon}
\newcommand{\Om}{\Omega}

\newcommand{\LIP}{\operatorname{LIP}}

\newcommand{\Ha}{\mathcal{H}}
\newcommand{\ud}{\mathrm {d}}
\newcommand{\AMV}{\operatorname{AMV}}
\newcommand{\id}{\operatorname{id}}
\newcommand{\M}{\mathcal{M}}

\definecolor{blau}{rgb}{0.1,0.0,0.9}

\newcommand{\red}{\color{red}}
\newcounter{komcounter}
\numberwithin{komcounter}{section}

\newcommand{\kom}[1]{}
\renewcommand{\kom}[1]{$\backslash$ {\bf  \red #1}$\backslash$}

\title{Asymptotically mean value harmonic functions in doubling metric measure spaces}

\author[Tomasz Adamowicz]{Tomasz Adamowicz{\small$^1$}}
\address{T.A.: Institute of Mathematics, Polish Academy of Sciences,
\'Sniadeckich 8, Warsaw, 00-656, Poland\/}
\email{tadamowi@impan.pl}

\author[Antoni Kijowski]{Antoni Kijowski}
\address{A.K.: Analysis on Metric Spaces Unit, Okinawa Institute of Science and Technology Graduate University, Okinawa, 904-0495, Japan\/}
\email{antoni.kijowski@oist.jp}

\author[Elefterios Soultanis]{Elefterios Soultanis{\small$^2$}}
\address{E.S.: Radboud University, IMAPP, Heyendaalseweg 135, 6525 AJ Nijmegen, The Netherlands\/}
\email{elefterios.soultanis@gmail.com}

\begin{document}
\maketitle

\footnotetext[1]{T. Adamowicz was supported by a grant of National Science Center, Poland (NCN),
 UMO-2017/25/B/ST1/01955.}
\footnotetext[2]{E. Soultanis was supported by Swiss National Foundation grant  no.  182423.}
\begin{abstract}
We consider functions with an asymptotic mean value property, known to characterize harmonicity in Riemannian manifolds, in doubling metric measure spaces. We show that the strongly amv-harmonic functions are H\"older continuous for any exponent below one. More generally, we define the class of functions with finite amv-norm and show that functions in this class belong to a fractional Haj\l asz--Sobolev space and their blow-ups satisfy the mean-value property. Furthermore, in the weighted Euclidean setting we find an elliptic PDE satisfied by amv-harmonic functions.
\end{abstract}

\setcounter{tocdepth}{1}
\tableofcontents

\section{Introduction}

\subsection{Overview}
A central feature of harmonic functions on Euclidean domains is the mean value property. Although the Euclidean structure is crucial for its validity, the mean value property can be expressed using only the metric and the measure, in terms of having vanishing \emph{$r$-laplacian}. The $r$-laplacian of a function $u\in L^1_{loc}(X)$ on a metric measure space $X=(X,d,\mu)$ is defined by

\begin{equation}\label{eq:rlap}
\Delta_{r}^\mu u(x)=\Delta_ru(x)=\frac{u_{B_r(x)}-u(x)}{r^2},\quad x\in X,
\end{equation}
where $u_{B_r(x)}$ stands for the mean-value of $u$ over a ball $B_r(x)$. The classical mean value property states that, in a Euclidean domain $\Omega$, a harmonic function $u$ satisfies $\Delta_r u(x)=0$, for all $0<r<\dist(x,\partial\Omega).$ Functions on a domain of a metric measure space with this property are called \emph{mean value harmonic} (mv-harmonic), see Definition \ref{def:mv}.

It turns out that harmonic functions rarely enjoy the mean value property outside the Euclidean setting and instead satisfy an \emph{asymptotic mean value property}, where the pointwise limit for $r$ approaching $0$ in \eqref{eq:rlap} vanishes. For example, this is the case of smooth Riemannian manifolds~\cite[Proposition 4.2]{aks2}, whereas the mean value property for harmonic functions on manifolds is known to hold only on the so-called harmonic manifolds. The Lichnerowicz conjecture, proven for manifolds  of dimension 2-5, characterizes harmonic manifolds as either flat or rank-one symmetric, see Example 4 in~\cite{agg} and references therein.


Apart from the classical setting, the $r$-laplacian also arises in approximation problems of Riemannian manifolds by graphs~\cite{bik}, and the mean value property plays a role in geometric group theory in Kleiner's proof of Gromov's polynomial growth theorem~\cite{kleiner} (see also \cite{mi-te, mi-te2} for a notion of amv-laplacian on metric measure spaces). Furthermore, \cite[Theorem 1]{co} indirectly relates amv-harmonicity to functions with bounded variation. Namely, the result characterizes $C^1$-minimal surfaces $S$ by observing that a certain piecewise constant function $f_S$ is pointwise amv-harmonic in the sense that $\Delta_r f_S\to 0$ pointwise, as $r~\to~0$ on $S$. Moreover, the proof of \cite[Theorem 1]{co} uses the relation between the amv-harmonic operator and a nondegenerate $1$-Laplacian  $\div(\nabla/ \sqrt{1+|\nabla|^2})$. The notion of $r$-laplacian also appears in the context of $p$-harmonic type equations and the stochastic games~\cite{mpr, mpr2, arlo}. In the setting of Carnot groups the $r$-Laplacian and its relations to subelliptic harmonic functions have been studied, for instance, in~\cite{adw, flm, fp}.

In this paper we consider \emph{asymptotically mean value harmonic} (amv-harmonic) functions on doubling metric measure spaces $X=(X,d,\mu)$ based on the asymptotic vanishing of the mv-laplacians $\Delta_r$ as $r\to 0$. In this generality the operators $\Delta_r$ are non-symmetric and there is no yet the existence theory for amv-harmonic functions. Nevertheless, some regularity results hold, see Theorems \ref{thm:holder} and \ref{thm:mvlip}. Additionally we obtain that amv-harmonicity translates to the mean value property of blow-ups, cf. Theorem \ref{thm:blowup}.

We demonstrate in the model case of weighted Euclidean spaces that in the limit $r\to 0$ the mv-laplace operators $\Delta_r$ give rise to a non-symmetric second order operator, leading to a characterization of amv-harmonicity in terms of an elliptic second order PDE, cf. Theorem \ref{thm:weight}.

	
The current manuscript is the first of two where we consider amv-harmonicity in nonsmooth settings. In the second part \cite{aks2} we prove e.g. that in non-collapsed RCD-spaces Cheeger harmonic functions are weakly amv-harmonic, and that in Carnot groups weak and strong amv-harmonicity equivalently characterize subelliptic harmonic functions. We moreover obtain Blaschke--Privaloff--Zaremba type theorems in step 2 Carnot groups and Alexandrov surfaces.
	

\subsection{Amv-harmonic functions and their regularity}
The following two notions are based on two ways to interpret the limit $\displaystyle \lim_{r\to 0}\Delta_ru=0$. 
\begin{definition}[Weak and strong AMV-harmonicity]\label{main-def1}
Let $\Omega\subset X$ be a domain in a metric measure space $X=(X,d,\mu)$. A function $u\in L^1_{loc}(\Omega)$ is \emph{strongly amv-harmonic} in $\Omega$ if
	\[
	\lim_{r\to 0}\|\Delta_r u\|_{L^\infty(K)}=0 \ \textrm{ for any compact }K\subset \Omega,
	\]
and \emph{weakly amv-harmonic} if
	\begin{equation*}
		\lim_{r\to 0}\int_X\varphi\Delta_ru\ud\mu=0\ \textrm{ for any }\varphi\in \LIP_c(\Omega).
	\end{equation*}
\end{definition}

Our first main result shows that strongly amv-harmonic functions are H\"older continuous for any exponent below one. 

\begin{theorem}\label{thm:holder}
	Let $\Omega\subset X$ be a domain in a complete doubling metric measure space. Then any strongly amv-harmonic function on $\Omega$ is locally $\alpha$-H\"older continuous for any $\alpha\in (0,1)$.
\end{theorem}


Since the operators $\Delta_r$ are not (asymptotically) symmetric and there is no meaningful limit operator defined on a large enough function space, in order to prove Theorem \ref{thm:holder} we consider instead the natural second order Sobolev-type space
\begin{equation}\label{eq:amvdomain}
	\AMV^{p}(X):=\left\{ u\in L^p(X):\ \|u\|_{\AMV^{p}}<\infty\right\},
\end{equation}
where $1\le p\le \infty$ and
\begin{equation}\label{eq:amvnorm}
	\|u\|_{\AMV^{p}}:=\limsup_{r\to 0}\|\Delta_ru\|_{L^p(X)}
\end{equation}
is the amv-norm of $u$. The class $\AMV^{p}_{loc}(X)$ is defined as the space of functions $u\in L^p_{loc}(X)$ for which $\displaystyle \limsup_{r\to 0}\|\Delta_ru\|_{L^p(K)}<\infty$ for every compact $K\subset X$. The following embedding result relates $\AMV^{p}_{loc}(X)$ to the fractional Haj\l asz--Sobolev spaces $M^{\alpha,p}_{loc}(X)$, see Section \ref{sec:holder}.

\begin{theorem}\label{thm:frachaj}
	Let $\Omega\subset X$ be an open subset of a proper locally doubling metric measure space $X=(X,d,\mu)$, and let $u\in \AMV_{loc}^p(\Omega)$. Then $u\in M_{loc}^{\alpha,p}(\Omega)$ for every $0<\alpha<1$. 
\end{theorem}

Together with the Morrey embedding theorem for fractional Haj\l asz--Sobolev spaces (see Proposition \ref{prop:fracmorrey} and \cite[Corollary 1.4]{yangd}), Theorem \ref{thm:frachaj} implies Theorem \ref{thm:holder} and the stronger Corollary \ref{cor:holder}. Theorem \ref{thm:frachaj} is not quantitative, and gives no explicit bounds on the fractional Haj\l asz--Sobolev gradient in terms of the amv-norm. Its proof relies on a refined averaging technique and a bootstrapping argument, see Section \ref{sec:mv} and Proposition \ref{prop:improve}. Refined averaging yields a Lipschitz regularization and moreover leaves mv-harmonic functions (locally) invariant. This observation directly yields local Lipschitz continuity of mv-harmonic functions, improving Corollaries 5.1 and 5.2 in~\cite{agg}, where local Lipschitz regularity is obtained, respectively, for uniform measures and measures satisfying the annular decay condition and under the assumption that the underlying space supports the $(1,p)$-Poincar\'e inequality.

\begin{theorem}\label{thm:mvlip}
	Let $\Omega\subset X$ be an open subset of a complete locally doubling metric measure space $X=(X,d,\mu)$, and $u\in L^1_{loc}(\Omega)$ an mv-harmonic function on $\Omega$. Then $u$ is locally Lipschitz and satisfies the bound
	\begin{equation}\label{eq:mvlip}
		\LIP(u|_{B_r(x_0)})\le \frac{C}{r}\inf_{c\in\R}\vint_{B_{3r}(x_0)}|u-c|\ud\mu
	\end{equation}
	whenever $\bar B_{3r}(x_0)\subset \Omega$.
\end{theorem}

Concerning the "size" of the class of mv-harmonic functions, in Section \ref{sec:mv} we also discuss the dimension bounds on $\harm^m(X)$, the space of mv-harmonic functions with $m$-growth on space $X$. Thus, we obtain a counterpart of results known for Riemannian manifolds, Alexandrov and RCD-spaces (see Section~\ref{sec:polygrowth} for the details).

\begin{prop}\label{prop:polygrowth}
	Let $X=(X,d,\mu)$ be a complete doubling metric measure space with doubling exponent $Q:=\log_2C_\mu>1$, and suppose $\mu$ has $\alpha$-annular decay, cf.~\eqref{eq-def-an-dec}. Then, for any $m>0$, we have that 
	\begin{align*}
	\dim \harm^m(X)\le Cm^{Q-\alpha},
	\end{align*}
	where the constant $C=C(Q,\alpha)$ depends only on $Q$ and $\alpha$.
\end{prop}

Though the mean value property (cf. Definition \ref{def:mv}) is very rare outside the Euclidean setting, we show that they arise as blow-ups of $\AMV^{p}$-functions. 

\begin{theorem}\label{thm:blowup}
	Let $\Omega\subset X$ be an open subset of a proper locally doubling length space $X=(X,d,\mu)$, and let $1<p<\infty$. Suppose $u\in M^{1,p}_{loc}(\Omega)\cap \AMV^p_{loc}(\Omega)$ and $(r_k)$ is a positive sequence converging to zero. Then for $\mu$-almost every $x\in \Omega$, any approximate tangent map $u_\infty:X_\infty\to \R$ at $x$, subordinate to a subsequence of $(r_k)$, is mv-harmonic.
\end{theorem}

This can be seen as an infinitesimal connection between amv- and mv-harmonicity. In light of the fact that in rectifiable spaces blow-ups are linear (and thus satisfy the mean value property automatically), Theorem \ref{thm:blowup} may serve as an obstruction to having many amv-harmonic functions on metric spaces that are too irregular.

\subsection{Amv-harmonicity in weighted Euclidean domains} As a toy model, we consider weighted Euclidean domains. Let $\|\cdot\|$ be a norm on $\R^n$ and $\Omega\subset \R^n$ a domain. We denote by $\Ha^n$ the Hausdorff $n$-measure, taken with respect to the metric induced by $\|\cdot\|$, restricted to $\Omega$. By a \emph{weight} on $\Omega$ we mean a strictly positive locally Lipschitz function $w:\Omega\to (0,\infty)$. In what follows, we consider the corresponding weighted Euclidean domain $\Omega_w:=(\Omega,\|\cdot\|, w\ud\Ha^n)$.

The $r$-laplacian $\Delta_r^w$ on $\Omega_w$ may fail to be asymptotically self-adjoint. This is explained more precisely further in this article, see Section \ref{sec-dbl} and Section \ref{sec:eucl}. Nevertheless, weak amv-harmonicity in $\Omega_w$ can be characterized by an elliptic PDE. Consider the unbounded operator
\begin{equation}\label{w-div-oper}
	L_wu:=\frac{1}{2}\div(M\nabla u)+\frac{1}{w}\langle \nabla w,M\nabla u\rangle,
\end{equation}
where $\langle \cdot,\cdot\rangle$ is the standard Euclidean inner product, and $M\in\R^{n\times n}$ denotes the positive definite \emph{matrix of second moments} of the unit ball $B^n\subset \R^n$ of $\|\cdot\|$ given by
\[
m_{ij}:=\vint_{B^n}y_iy_j\ud\Ha^n(y),\quad 1\le i,j\le n,
\]
see \cite[Sections 4.1-4.3]{ak}.

\begin{theorem}\label{thm:weight}
	Let $u\in W^{1,2}_{loc}(\Omega)$ for a domain $\Om\subset \R^n$. Then the following are equivalent.
	\begin{itemize}
		\item[(1)] $L_wu=0$ in $\Omega$ in the weak sense.
		\item[(2)] $u$ is weakly amv-harmonic in $\Omega_w$. 
		\item[(3)] $u\in AMV_{loc}^2(\Omega_w)$ and $\Delta_r^wu\to 0$ in $L^2_{loc}(\Omega_w)$ as $r\to 0$.
	\end{itemize}
	Moreover, if in addition $w\in C^\infty(\Omega)$, then (1), (2) and (3) are equivalent to
	\begin{itemize}
		\item[(4)] $u$ is strongly amv-harmonic in $\Omega_w$.
	\end{itemize}
\end{theorem}
We refer to~\cite{mi-te} and \cite{bos} for some related results. The non-trivial implications in Theorem \ref{thm:weight} boil down to regularity theory of elliptic PDEs, and a characterization of ${\rm AMV}_{loc}^2(\Omega_w)$ as a second order Sobolev space, see Theorem \ref{thm:weight2} in Section \ref{sec:eucl}.

In view of classical existence theory, weakly and strongly amv-harmonic functions are abundant under the hypotheses of Theorem \ref{thm:weight} and their further studies in the weighted setting could be interesting per se. In contrast, mean value harmonic functions (with $\Omega=\R^n$) satisfy an overdetermined system of PDEs --  in particular, if $\|\cdot\|$ is an $l^p$-norm ($p\ne 2$), mv-harmonic functions on (unweighted) $\R^2$ form a finite dimensional space; see \cite{ak}.

\section{Preliminaries}

\subsection{Notation and conventions}\label{subs-notation}
Throughout this article, a metric measure space $X=(X,d,\mu)$ is a separable metric space $(X,d)$ equipped with a Radon measure $\mu$ that is finite and non-trivial on balls, i.e. $0<\mu(B)<\infty$ for all balls $B\subset X$.

Given a subset $F\subset X$ of a metric space and $r>0$, we denote
\[
N_r(F)=\{ x\in X:\ \dist(x,F)<r\} \textrm{ and }\overline N_r(F)=\{x\in X:\ \dist(x,F)\le r\}
\]
the open and closed $r$-neighbourhood of $F$ (note that $\overline N_r(F)$ need not be the closure of $N_r(F)$ unless $X$ is a length space). For $x\in X$, we denote by $B_r(x):=N_r(\{x\})$ and $\overline B_r(x):=\overline N_r(\{x\})$,  respectively, an open and closed ball centered at $x$ with radius $r$.

The Lipschitz constant of a map $f:(X, d_X) \to (Y, d_Y)$ between metric spaces is
\begin{equation*}
\LIP(f):=\sup_{x\ne y}\frac{d_Y(f(x),f(x))}{d_X(x,y)},
\end{equation*}
and the pointwise Lipschitz slope is
\begin{align*}
\Lip f(x):= \limsup_{r\to 0}\sup_{0<d(y,x)<r}\frac{d_Y(f(x),f(y))}{r},\quad x\in X.
\end{align*}
If $x\in X$ is an isolated point, then $\Lip f(x)=0$.

Given $u\in L_{loc}^1(X)$ and a $\mu$-measurable set $A\subset X$ with $\mu(A)>0$, 
\[
u_{A}=\vint_A u\ud\mu=\frac{1}{\mu(A)}\int_Au\ud\mu
\]
denotes the  average of $u$ over $A$.

\subsection{Doubling measures and averaging operators}\label{sec-dbl}

A measure $\mu$ on a separable metric space $X$ is called \emph{locally doubling} if, for every compact $K\subset X$, there exists $r_K>0$ and a constant $C_K>0$, such that 
\begin{equation}\label{eq:doubl}
\mu(B_{2r}(x))\le C_K\mu(B_r(x))
\end{equation}
for every $x\in K$ and $0<r\le r_K$. If $\mu$ is locally doubling, then for every compact $K\subset X$ there exists a constant $C>0$ for which 
\begin{equation}\label{eq:doublexp}
\frac{\mu(B_{r}(y))}{\mu(B_R(x))}\ge C\left(\frac{r}{R}\right)^Q,\quad x\in K,\ y\in B_R(x),\ 0<r\le R\le r_K,
\end{equation}
where $Q=\log_2C_K$. If the constant $C_K$ can be chosen independently of the set $K\subset X$, and $r_K=\infty$, then we say that $\mu$ is \emph{doubling}, denote $C_K=C_\mu$ and the number $Q=\log_2C_\mu$ is called the \emph{doubling exponent} of $\mu$.

The following definition is due to Buckley, see~\cite[Section 1]{buc}, and is stronger than the doubling condition.

\begin{definition}\label{defn-an-dec}
 Let $(X, d, \mu)$ be a metric measure space with a doubling measure $\mu$. We say that $X$ satisfies the \emph{$\alpha$-annular decay property} with some $\alpha\in(0,1]$ if there exists $A\geq 1$ such that for all $x\in X$, $r>0$ and $\ep\in(0,1)$ it holds that
 \begin{equation}\label{eq-def-an-dec}
  \mu\left(B(x,r)\setminus B(x,r(1-\ep))\right)\leq A \ep^\alpha \mu(B(x,r)).
 \end{equation}
If $\alpha=1$, then we say that $X$ satisfies the \emph{strong annular decay property}.
\end{definition}

Examples of spaces with strong annular decay property include geodesic metric spaces with uniform measures and Heisenberg groups $\mathbb{H}^n$ equipped with a left-invariant Haar measures. By \cite[Corollary 2.2]{buc}, a length space with a doubling measure has the $\alpha$-annular decay property for some $\alpha\in(0,1]$ with $\alpha$ depending only on a doubling constant of the measure. (In fact, it is enough for a metric measure space to be the so-called $(\alpha, \beta)$-chain space to conclude that the space has the $\delta$-annular decay property, see Theorem 2.1 in \cite{buc}).

Let $(X,d,\mu)$ be a metric measure space and $r>0$. Given a locally integrable function $u\in L^1_{loc}(X)$, we denote by
\[
A_r^\mu u(x)=A_r u(x)=\vint_{B_r(x)}u\ud\mu,\quad x\in X,
\]
the $r$-average function of $u$. Note that $A_ru(x)=u_{B_r(x)}$.
 
We will use the two notations interchangeably, depending on whether we want to view the average as a number, or an operator on a function space. Indeed, the function $A_ru:X\to\R$ is measurable, and $A_r$ defines a bounded linear operator $A_r:L^1(X)\to L^1(X)$ if and only if $a_r\in L^\infty(X)$, where
\[
a_r(x)=\int_{B_r(x)}\frac{\ud\mu(y)}{\mu(B_r(y))},\quad x\in X.
\]
Moreover, in this case the operator norm satisfies $\|A_r\|_{L^1\to L^1}=\|a_r\|_{L^\infty}$, see~\cite[Theorem 3.3]{Aldaz18}. This is true in particular when $\mu$ is a doubling measure.  On the other hand, it is true that by the Lebesgue differentiation theorem
\begin{align*}
u(x)=\lim_{r\to 0}A_ru(x)\,\textrm{ for almost every }x\in X,
\end{align*}
if $\mu$ is infinitesimally doubling, cf.~\cite[Remark 3.4.29]{HKST}.

If $X$ is doubling as a metric space, then there exists $C>0$ so that $\|A_r\|_{L^p\to L^p}\le C$ for every $r>0$ and every $1\le p<\infty$, cf.~\cite[Theorem 3.5]{Aldaz18}. However, $A_r$ is not a self-adjoint operator; the formal adjoint $A_r^*$ of $A_r$  is given by
\begin{equation*}
(A_r^\mu)^*u(x)=A_r^*u(x):=\int_{B_r(x)}\frac{u(y)\ud\mu(y)}{\mu(B_r(y))},\quad x\in X,
\end{equation*}
for $u\in L^1_{loc}(X)$. Indeed, a direct computation using the Fubini theorem yields that
\begin{equation}\label{eq:alaGreen}
\int_X vA_ru\ud\mu=\int_XuA_r^*v\ud\mu,\quad u\in L^p(X),\ v\in L^q(X),
\end{equation}
where $1/p+1/q=1$.

We may express the $r$-laplacian using the averaging operator as
\begin{align*}
\Delta_ru=\frac{A_ru-u}{r^2},\quad u\in L^1_{loc}(X).
\end{align*}
Next, we denote by
\begin{equation}\label{eq:adjointlaplace}
\Delta_r^*u:=\frac{A_r^*u-u}{r^2},\quad  u\in L^1_{loc}(X),
\end{equation}
the formal adjoint of the $r$-laplacian. Note that if $A_r:L^p(X)\to L^p(X)$ is bounded, then $\Delta_r:L^p(X)\to L^p(X)$ and $\Delta_r^*:L^q(X)\to L^q(X)$ are both bounded, where $1/p+1/q=1$.

\begin{remark}\label{rmk: Green}
 Identity \eqref{eq:alaGreen} together with~\eqref{eq:adjointlaplace} imply the following Green-type formula for  $u\in L^p(X),\ v\in L^q(X)$ where $1/p+1/q=1$:
 \[
  \int_X v \Delta_ru\ud\mu=\int_X u \Delta_r^*v\ud\mu.
 \]
The asymptotic behaviour of $\Delta_r-\Delta_r^*$ as $r\to 0$ plays an important role in~\cite{aks2}. 
\end{remark}

\begin{remark}\label{rmk:domain}
While most of our results will be formulated for metric measure spaces, they encompass the case of an open set $\Omega \subset X$ in the introduction. Indeed, an open subset $\Omega\subset X$ of a metric measure space can be regarded as a metric measure space $\Omega=(\Omega,d|_\Omega,\mu|_\Omega)$. In particular, \emph{if $X$ is locally doubling, then $\Omega$ is locally doubling.}
\end{remark}

\section{Refined averaging and mean value harmonic functions}\label{sec:mv}

In this section we consider mean value harmonic functions. In particular, we show their local Lipschitz regularity assuming merely the doubling property of the underlying measure, see Theorem \ref{thm:mvlip}. We also prove a dimension bound on the space of mv-harmonic functions with polynomial growth in the spirit of the celebrated results of Colding--Minicozzi~\cite{cm, cm2} confirming Yau's conjecture, see Proposition~\ref{prop:polygrowth}. Our approach emphasizes the role of the averaging operators. 

\begin{definition}\label{def:mv}
	Let $X=(X,d,\mu)$ be a metric measure space. We say that a function $u\in L^1_{loc}(X)$ is \emph{mean value harmonic} (or has the \emph{mean value property}) if 
	\[
	\Delta_ru=0\textrm{ on }K
	\]
	for any compact set $K\subset X$ and $r<r_K:=\sup\{\rho>0:\ \overline N_r(K)\textrm{ is compact} \}$.
\end{definition}
If $X$ is a complete doubling metric measure space with doubling exponent $Q$, and $\Omega\subset X$ a domain, we denote by $\Ha(\Omega)$ the space of mv-harmonic functions on $\Omega$. Note that, if $u\in \Ha(\Omega)$, then
\begin{align}\label{eq:mvprop}
u(x)=u_{B_r(x)},\quad x\in \Omega,\ r<\dist(x,X\setminus \Omega).
\end{align}
The mean value property \eqref{eq:mvprop} has strong implications. In particular, a Harnack inequality holds on balls: if $u\in \harm(\Omega)$ is non-negative and $\overline B_{4r}(x_0)\subset\Omega$, then
\begin{equation}\label{eq:ballharnack}
\sup_{B_r(x_0)}u\le C\inf_{B_r(x_0)}u,
\end{equation}
where the constant depends only on the doubling constant of $\mu$, see~\cite[Section 4]{agg}. A standard argument then gives the Harnack inequality for non-negative mv-harmonic functions, cf. \cite[Proposition 4.2]{agg}. We refer to~\cite{agg} for further properties of mv-harmonic functions in metric measure spaces. In the next section we see that the mean value property \eqref{eq:mvprop} also yields more regularity than one might initially expect.

\subsection{Local Lipschitz continuity of mean value-harmonic functions}\label{subs:loc-lip}
If a measure $\mu$ has the $\alpha$-annular decay property (see Definition~\ref{defn-an-dec}), then \cite[Theorem 4.2]{agg} shows that mv-harmonic functions are $\alpha$-H\"older continuous. In this section we will prove that in fact, mv-harmonic functions are \emph{Lipschitz} continuous even when the measure does not satisfy the annular decay condition. The idea of the proof is to consider a refined averaging process, wherein we average over the radius as well as the space variable. Throughout the subsection we use the abbreviation
\[
u_r:=A_ru
\]
for $u\in L^1_{loc}(X)$ and $r>0$. Given such a function we define
\begin{align}\label{eq:refav}
A^ru(x)=u^r(x):=\frac 2r\int_{r/2}^ru_t(x)\ud t,\quad x\in X.
\end{align}

According to our best knowledge this type of the refined averaging operator has not yet been considered in the literature. The following elementary lemma will play a crucial role in the estimates proving that $A^ru$ is locally Lipschitz. We use the notation 
\[
A_{r,R}(x)=\bar B_R(x)\setminus B_r(x)
\]
for $x\in X$ and $r\le R$, with the convention that $B_r(x)=\varnothing$ for $r\le 0$ and $A_{r,R}(x)=\varnothing$ if $r>R$.

\begin{lemma}\label{elementary}
	Let $f\in L^1_{loc}(X)$ be a nonnegative function and $x\in X$. Let $0\le r\le R<\infty$, and $-\infty<d_1\le d_2<\infty$. 
	Then
	\[
	\int_r^R\int_{A_{t+d_1,t+d_2}(x)}f\ud\mu\ \ud t\le (d_2-d_1)\int_{A_{r+d_1,R+d_2}(x)}f\ud\mu.
	\]	
\end{lemma}

\begin{proof}
	Let us consider two cases. If $R+d_1\le r+d_2$, then $R-r\le d_2-d_1$ and we have the trivial estimate
	\[
	\int_r^R\int_{A_{t+d_1,t+d_2}(x)}f\ud\mu\ \ud t\le \int_r^R\int_{A_{r+d_1,R+d_2}(x)}f\ud\mu\ud t\le (d_2-d_1)\int_{A_{r+d_1,R+d_2}(x)}f\ud \mu.
	\]
	Next, assume that $R+d_1>r+d_2$ and define an auxiliary function $g:\R\to \R$ as follows: $g(t)=0$ for $t\le 0$ and 
	\[
	g(t)=\int_{B_t(x)}f\ud\mu,\quad t\ge 0.
	\]
	Then $g$ is a nondecreasing function and it holds that 
	\begin{align*}
	&\int_r^R\int_{A_{t+d_1,t+d_2}(x)}f\ud\mu\ \ud t=\int_r^R[g(t+d_2)-g(t+d_1)]\ud t=\int_{r+d_2}^{R+d_2}g\ud t-\int_{r+d_1}^{R+d_1}g\ud t\\
	=&\int_{R+d_1}^{R+d_2}g\ud t-\int_{r+d_1}^{r+d_2}g\ud t\le (d_2-d_1)g(R+d_2)-(d_2-d_1)g(r+d_1)\\
	=&(d_2-d_1)\int_{A_{r+d_1,R+d_2}(x)}f\ud \mu.
	\end{align*}
\end{proof}

\begin{prop}\label{lipaprox}
	Let $X=(X,d,\mu)$ be a locally doubling metric measure space, $u\in L^1_{loc}(X)$. Let $K\subset X$ be compact and $r_K>0$ such that \eqref{eq:doubl} holds for $r<r_K$. Then, for any $r<r_K$, the function $u^r$ is Lipschitz on $K$ and, for any $c\in \R$, satisfies the Haj\l asz type estimate
	\begin{align}\label{eq:hajtype}
	|u^r(x)-u^r(y)|\le \frac{Cd(x,y)}{r}\left(\vint_{B_{2r}(x)}|u-c|\ud\mu+\vint_{B_{2r}(y)}|u-c|\ud\mu\right),\quad x,y\in K,\ d(x,y)<r,
	\end{align}
	where the constant $C$ depends only on the doubling constant of $\mu$ on $K$. In particular
	\begin{align*}
	\LIP(u^r|_{B_r(x_0)})\le \frac{C}{r}\vint_{B_{3r}(x_0)}|u-c|\ud\mu
	\end{align*}
	whenever $\bar B_{3r}(x_0)\subset X$ is compact.
\end{prop}

\begin{proof}[Proof of Proposition \ref{lipaprox}]
	Let us first prove the Lipschitz estimate assuming \eqref{eq:hajtype}. Let $K=B_r(x_0)$, $c\in R$ and assume that $x,y\in B_{r}(x_0)$ with $d(x,y)<r$. Then \eqref{eq:hajtype} together with the doubling condition directly yield
	\begin{align*}
	|u(x)-u(y)|\le \frac{Cd(x,y)}{r}\left(\vint_{B_{2r}(x)}|u-c|\ud\mu+\vint_{B_{2r}(y)}|u-c|\ud\mu\right)  \le \frac{Cd(x,y)}{r}\vint_{B_{3r}(x_0)}|u-c|\ud\mu.
	\end{align*}
	If $d(x,y)\ge r$, then $d(x,x_0)+d(y,x_0)\le 2d(x,y)$, and therefore
	\begin{align*}
	|u^r(x)-u^r(y)|\le &|u^r(x)-u^r(x_0)|+|u^r(y)-u^r(x_0)|\le \frac{C(d(x,x_0)+d(y,x_0))}{r}\vint_{B_{3r}(x_0)}|u-c|\ud\mu\\
	\le & \frac{2Cd(x,y)}{r}\vint_{B_{3r}(x_0)}|u-c|\ud\mu.
	\end{align*}
	
	Thus it suffices to prove \eqref{eq:hajtype}. Let $K$ and $r$ be as in the claim and denote by $C_K$ the doubling constant of $\mu$ in $K$. Given $x,y\in K$ with $d:=d(x,y)\le r$ and $r/2\le t< r$, we have that for the symmetric difference of two balls it holds 
	\[
	B_t(x)\triangle B_t(y)\subset A_{t-d,t+d}(x)\subset N_{2r}(K).
	\]
	Recall that $N_{2r}(K)$ stands for the $2r$-neighbourhood of $K$. This, formula (7) in~\cite{agg}  and the doubling property of $\mu$ on $K$ yield, for any $c\in \R$, the estimate 
	\begin{align*}
	|u_t(x)-u_t(y)|&=\left|\vint_{B_t(x)}(u-c)d\mu-\vint_{B_t(y)}(u-c)d\mu\right|\\
	&\le\frac{\mu(B_t(x)\triangle B_t(y))}{\mu(B_t(x))}\vint_{B_t(y)}|u-c|d\mu + \frac{1}{\mu(B_t(x))}\int_{B_t(x)\triangle B_t(y)}|u-c|d\mu\\
	&\le \frac{C_K^2\mu(A_{t-d,t+d}(x))}{\mu(B_r(x))}\vint_{B_r(y)}|u-c|\ud\mu+\frac{C_K}{\mu(B_r(x))}\int_{A_{t-d,t+d}(x)}|u-c|\ud\mu.
	\end{align*}
	By Lemma~\ref{elementary} -- applied twice ($f\equiv 1$ and $f=|u-c|$) with $d_1=-d$ and $d_2=d$ -- the above inequality implies
	\begin{align*}
	\int_{r/2}^r|u_t(x)-u_t(y)|\ud t \le &\, \frac{2C_K^2d\mu(A_{r/2-d,r+d}(x))}{\mu(B_r(x))}\vint_{B_r(y)}|u-c|\ud\mu+\frac{2C_Kd}{\mu(B_r(x))}\int_{A_{r/2-d,r+d}(x)}|u-c|\ud\mu\\
	\le &\, 2C_K^3d\left(\vint_{B_{2r}(y)}|u-c|\ud\mu+\vint_{B_{2r}(x)}|u-c|\ud\mu\right),
	\end{align*}
	where in the last inequality we also appeal to the doubling property~\eqref{eq:doublexp} of $\mu$ .
	Using this estimate we obtain
	\begin{align*}
	|u^r(x)-u^r(y)|\le \frac 2r \int_{r/2}^r|u_t(x)-u_t(y)|\ud t\le \frac{4C_K^3d(x,y)}{r}\left(\vint_{B_{2r}(y)}|u-c|\ud\mu+\vint_{B_{2r}(x)}|u-c|\ud\mu\right),
	\end{align*}
	proving \eqref{eq:hajtype}.
\end{proof}

\begin{proof}[Proof of Theorem \ref{thm:mvlip}]
	By considering the metric measure space $(\Omega,d,\mu|_{\Omega})$ we may assume that $\Omega=X$, cf. Remark \ref{rmk:domain}. Since $u$ is mv-harmonic, for every compact $K\subset X$ there exists $r_K>0$ so that $u=u_r$ on $K$ for all $r<r_K$. In particular, $u=u^r$ on $K$	whenever $r<r_K$. The Lipschitz continuity of $u$ and the estimate \eqref{eq:mvlip} then follows by Proposition~\ref{lipaprox}.
\end{proof}

\begin{remark}\label{rmk:nosublin}
	When $X$ is complete and $\mu$ is globally doubling, mv-harmonic functions satisfy $u=u_R$ on $X$, for \emph{any} $R>0$. Consequently \eqref{eq:mvlip} yields
	\begin{align*}
	\LIP(u|_{B_R(p)})\le \frac{C}{R}\vint_{B_{3R}(p)}|u-u(p)|\ud\mu, \quad p\in X.
	\end{align*}
\end{remark}
Remark \ref{rmk:nosublin} together with Harnack's inequality \eqref{eq:ballharnack} imply, in particular, that there are no non-constant mv-harmonic functions of sublinear growth. A function $u$ is said to have sublinear growth, if
\[
\limsup_{R\to \infty}\frac 1R \sup_{B_R(p)}|u|=0
\]
for some, and hence any $p\in X$.

The following observation can be considered as a counterpart of Cheng's result for harmonic functions with sublinear growth on complete manifolds with nonnegative Ricci curvature, see Corollary 1.5 in~\cite{pli}, and it is also related to the celebrated Phragmen-Lindel\"of theorem.
\begin{corol}
	Let $X=(X,d,\mu)$ be a doubling metric measure space. If $u$ is mv-harmonic and has sublinear growth, then it is constant.
\end{corol}
In the next section we discuss mean value harmonic functions with prescribed growth in more detail.

\subsection{Mean value-harmonic functions of polynomial growth}\label{sec:polygrowth}
The purpose of this section is to prove Proposition~\ref{prop:polygrowth}. In \cite{cm, cm2} Colding and Minicozzi proved a conjecture of Yau on the finite dimensionality of $\harm^m(M)$, the space of harmonic functions with polynomial growth of degree $m$, in a Riemannian $n$-manifold $M$ of non-negative Ricci curvature (here the harmonic functions are understood in the sense of solutions to the Laplace--Beltrami equation). Namely, they showed that,  the following bound holds:
\begin{equation*}
\dim \harm^m(M)\le C(n)m^{n-1}.
\end{equation*}
This result has been extended to Alexandrov and RCD-spaces, cf.~\cite{hua, hkx}. An argument of Li~\cite{li} uses the doubling property and the mean value inequality of subharmonic functions to obtain the estimate
\begin{equation}\label{eq:nonsharp}
\dim \harm^m(M)\le C(n)m^{Q}
\end{equation}
for manifolds with a measure satisfying \eqref{eq:doublexp} and a global Poincar\' e inequality. Thus estimate \eqref{eq:nonsharp} remains valid in the context of mv-harmonic functions on doubling spaces. 

In fact a modification of the same argument improves the bound \eqref{eq:nonsharp} if $\mu$ satisfies an annular decay property. We follow the strategy in~\cite{li}, see Lemmas \ref{lem:lemma1} and \ref{lem:lemma2}, and present the modifications needed for the proof of Proposition~\ref{prop:polygrowth}.

For $m>0$, let $\harm^m(X)$ denote the space of $u\in \harm(X)$ with growth rate at most $m$. A function $u\in \harm(X)$ is said to have \emph{growth rate at most $m$} if, for some (and hence all) $p\in X$, there exists $C>0$ so that
\begin{align*}\label{harm-poly-growth}
|u(x)|\le C(1+d_p(x))^m,\quad x\in X,
\end{align*}
where $d_p:X\to \R$ is the distance function $x\mapsto d(p,x)$.


Doubling measures on length spaces always satisfy an annular decay property for some $\alpha$, see the discussion following  Definition \ref{defn-an-dec}. Thus, Proposition \ref{prop:polygrowth} implies the following corollary.

\begin{corol}
	Let $(X, d, \mu)$ be a complete geodesic doubling metric measure space with $Q>1$. Then there exists $\delta>0$,  depending only on the doubling constant of $\mu$, so that
	\begin{align*}
	\dim \harm^m(X)\le C(Q)m^{Q-\delta}.
	\end{align*}
\end{corol}

Let $X=(X,d,\mu)$ be a complete doubling metric measure space, where $\mu$ has $\alpha$-annular decay. Given $R>0$, we define a bi-linear form 
\begin{align*}
A_R(u,v):=\int_{B_R(p)}uv\ud\mu, \quad u,v\in \harm(X),
\end{align*}
and note that $A_R$ is symmetric and positive semidefinite. It follows from the proof of \cite[Lemma 3.4]{hua} that, for any finite dimensional vector subspace $V\subset \harm(X)$, there exists a radius $R_0>0$ so that $A_R$ is an inner product on $V$ for every $R>R_0$. 

In order to prove Proposition~\ref{prop:polygrowth} we need the following auxiliary results. 

\begin{lemma}\label{lem:lemma1}
	Let $V$ be a $k$-dimensional linear subspace of $\harm^m(X)$. For any $p\in X$, $\beta>1$, $\delta>0$, $R_0>0$  there exists $R>R_0$ such that if $u_1,\ldots, u_k$ is an orthonormal basis for $V$ with respect to the inner product $A_{\beta R}$, 
	then
	\[
	\int_{B_R(p)}(u_1^2+\ldots+u_k^2)\,d\mu\geq \frac{k}{\beta^{2m+Q+\delta}}.
	\]
\end{lemma}
\begin{proof}
	The proof of~\cite[Lemma 2]{li} for manifolds carries over to the setting of metric measure spaces under our assumptions. See also \cite[Lemma 3.7]{hua}, where the lemma is proven in the setting of Alexandrov spaces, and \cite[Lemma 5.2]{hkx} for the formulation of the lemma in the $RCD^*(0,N)$ spaces.
\end{proof}

\begin{lemma}\label{lem:lemma2}
	Let $V$ be a $k$-dimensional linear subspace of $\harm^m(X)$. Then, there exists a constant $C=C(Q)$ such that for any base $u_1,\ldots, u_k$ of $V$, any $p\in X$, $R>0$ and any $\ep\in(0,\frac12)$ it holds that
	\[
	\int_{B_R(p)}(u_1^2+\ldots+u_k^2)\,d\mu\leq \frac{C}{\ep^{Q-\alpha}} \sup_{\tiny{ \sum_{i}^ka_i^2=1}}\int_{B_{(1+\ep)R}(p)} |a_1u_1+\ldots+a_ku_k|^2\ud\mu.
	\]
\end{lemma}
\begin{proof}
	
	The proof of Lemma~\ref{lem:lemma2} is a modification of the corresponding one of~\cite[Lemma 1]{li}. We follow closely the proof of Lemma 5.3 in~\cite{hkx}. Fix $q\in B_r(p)$ and define $V_q:=\{u\in V: u(q)=0\}$. The subspace $V_q\subset V$ is of $\codim V_q\leq 1$ (since, if $u,v\not\in V_q$, then $u-\frac{u(q)}{v(q)}v\in V_{q}$).
	
	There exists an orthogonal change of variables $A$ on $V$ such that $A(u_i):=v_i$ for all $i=1,\ldots,k$ with $v_i\in V_{q}$ for $i=2,\ldots, k$.
	
	We recall the relevant part of Proposition 3.1 in~\cite{agg} which states that, \emph{if $f\in \harm(X)$, then $F\circ f$ is subharmonic provided $F: f(X)\to \R$ is convex}. We apply this result for $f=v_1$ and $F(s)=s^2$ to obtain that
	\begin{align}\label{lem2-key-est}
	\sum_{i=1}^{k}u_i^2(q)= \sum_{i=1}^{k}v_i^2(q)  &=v_1^2(q)\leq \vint_{B_{(1+\ep)R-d_p(q)}(q)} v_1^2(z)\,d\mu(z) \nonumber \\
	&\leq \sup_{\tiny{\substack{a_i\in \R, i=1,\ldots, k \\ \sum_{i}a_i^2=1}}}\frac{1}{\mu(B_{(1+\ep)R-d_p(q)}(q))}\int_{B_{(1+\ep)R}(p)} \left|\sum_{i=1}^k a_iu_i(z)\right|^2\, d\mu(z).
	\end{align} 
	We apply~\eqref{eq:doublexp} to obtain that
	\[ 
	\frac{1}{\mu(B_{(1+\ep)R-d_p(q)}(q))}\leq \frac{C_{\mu}^2}{\mu(B_{(1+\ep)R}(p))}\left(\frac{(1+\ep)R}{(1+\ep)R-d_p(q)}\right)^Q.
	\]
	Hence, upon integrating \eqref{lem2-key-est}, we arrive at
	
	\begin{align}\label{eq:Liest}
	&\sum_{j=1}^kA_R(u_j,u_j)\\
	\le &\frac{C}{\mu(B_R(p))}\left(\int_{B_R(p)}\left(1+\ep-\frac{d_p(q)}{R}\right)^{-Q}\ud\mu(q) \right)\sup_{\tiny{ \sum_{i}^ka_i^2=1}}\int_{B_{(1+\ep)R}(p)} |a_1u_1+\ldots+a_ku_k|^2\ud\mu.\nonumber
	\end{align}
	
	Denote 
	\[
	f:[0,1]\to \R,\quad f(t)=(1+\ep-t)^{-Q},
	\]
	and note that the claim follows directly from \eqref{eq:Liest} and the estimate
	\begin{equation}\label{eq:decayestimate}
	\vint_{B_R(p)}f(d_p/R)\ud\mu\le \frac{C}{\ep^{Q-\alpha}}.
	\end{equation}
	To obtain this, note that $f$ is smooth and increasing, and thus 
	\begin{align*}
	\int_{B_R(p)}f(d_p/R)\ud\mu=\int_0^\infty \mu(B_R(p)\cap \{ f\circ (d_p/R)\ge \lambda \})\ud\lambda=\int_{-\infty}^1 f'(s)\mu(B_R(p)\cap \{d_p\ge sR\})\ud s
	\end{align*}
	after a change of variables $f(s)=\lambda$. The $\alpha$-annular decay implies
	\[
	\mu(B_R(p)\cap \{d_p\ge sR\})=\mu(A_{sR,R}(p))\le C(1-s)^\alpha\mu(B_R(p))
	\]
	and therefore
	\begin{align*}
	\vint_{B_R(p)}f(d_p/t)\ud\mu\le &C\int_{-\infty}^1f'(s)(1-s)^\alpha\ud s\le CQ\int_{-\infty}^1 (1+\ep-s)^{-Q-1+\alpha}\ud s\\
	=& \frac{CQ}{Q-\alpha}\frac{1}{\ep^{Q-\alpha}},
	\end{align*}
	establishing \eqref{eq:decayestimate}.
\end{proof}
\begin{proof}[Proof of Proposition \ref{prop:polygrowth}]
	Let $R$ be large enough, $\beta=1+\ep$ and $\ep=1/(2m)$. Let $\{u_1,\ldots, u_k\}$ be an orthonormal basis with respect to $A_{\beta R}$.
	Combining the estimates in Lemmas \ref{lem:lemma1} and \ref{lem:lemma2} we obtain
	\[
	\frac{k}{(1+\ep)^{2m+Q+\delta}}\le \sum_{j=1}^kA_{R}(u_j,u_j)\le \frac{C(Q)}{\ep^{Q-\alpha}},
	\]
	since
	\[
	\sup_{\tiny{ \sum_{i}^ka_i^2=1}}\int_{B_{(1+\ep)R}(p)} |a_1u_1+\ldots+a_ku_k|^2\ud\mu=1.
	\]
	Thus
	\[
	k\le C(Q)(1+1/(2m))^{2m+Q+\delta}(2m)^{Q-\alpha}\le Cm^{Q-\alpha}
	\]
	after letting $\delta\to 0$.
\end{proof}

\section{H\"older continuity of strongly amv-harmonic functions}\label{sec:holder}
In this section we prove Theorems \ref{thm:holder} and \ref{thm:frachaj}. Throughout the section $X=(X,d,\mu)$ denotes a locally compact and doubling metric measure space with doubling exponent $Q$. We define the sharp maximal function $\M^\#u$ of a locally integrable function $u$ on $X$ as
\begin{align*}
\M^\#_Ru(x):=\sup_{0<r<R}\vint_{B_r(x)}|u-u_{B_r(x)}|\ud\mu,\quad x\in X.
\end{align*}
By standard maximal function estimates we have that, for any compact $K\subset X$, 
\begin{align*}
\int_{K}(\M^\#_Ru)^p\ud\mu\le C\int_{N_R(K)}|u|^p\ud\mu
\end{align*}
with a constant $C$ depending only on $p$ an the local doubling constant of $\mu$ on $K$, as long as $R<r_K$ (cf. \eqref{eq:doubl}). We use this observation in the proof of the next proposition without explicit mention.
\smallskip

We now briefly recall fractional Haj\l asz--Sobolev spaces and their local variants.
\begin{definition}\label{def:lochaj}
	Let $(X,d,\mu)$ be a metric measure space, and $1\le  p\le \infty$,  $0<\alpha\le 1$. The \emph{local fractional Haj\l asz-Sobolev space} $M^{\alpha,p}_{loc}(X)$ consists of all Borel functions $u\in L^p_{loc}(X)$ with the following property: there exists a null set $N\subset X$ and, for every compact $K\subset X$, a non-negative function $g_K\in L_{loc}^p(X)$ and $r_K>0$ such that 
	\begin{align}\label{eq:lochaj}
	|u(x)-u(y)|\le d(x,y)^\alpha[g_K(x)+g_K(y)],\quad x,y\in K\setminus N,\ d(x,y)<r_K.
	\end{align}
	If we can choose $g_K=g\in L^p(X)$ independently of $K$, and $r_K=\infty$, we say that $u\in M^{\alpha,p}(X)$, and call $g$ a Haj\l asz gradient of $u$.
\end{definition}

Fractional Sobolev--Haj\l asz spaces on connected doubling spaces (and more generally RD-spaces, cf. \cite{kyz}) coincide with a Triebel--Lizorkin space for appropriate parameters. For more details on the connection between Sobolev--Haj\l asz, Triebel--Lizorkin and Besov spaces we refer to \cite{kyz,kyz2}. Here we record the following Morrey embedding theorem for fractionl Sobolev--Haj\l asz spaces, see \cite[Corollary 1.4]{yangd}.

\begin{prop}[Fractional Morrey embedding]\label{prop:fracmorrey}
Let $B\subset X$ be a compact ball in a doubling metric measure space so that \eqref{eq:doublexp} holds in $B$ with exponent $Q$. Let $0<\alpha\le 1$, $0<p<\infty$ and suppose that $\alpha p>Q$. Then there exists a constant $C>0$ so that every $u\in M^{\alpha,p}(B)$ satisfies
\begin{align*}
|u(x)-u(y)|\le Cd(x,y)^{\alpha-Q/p}\left(\vint_{B}g^p\ud\mu\right)^{1/p},\quad x,y\in B,
\end{align*}
whenever $g\in L^p(B)$ is a Haj\l asz gradient of $u$.
\end{prop}

Proposition \ref{prop:fracmorrey} also follows from \cite[Theorem 8.7]{haj} by considering the metric space $(X,d^\alpha,\mu)$. We omit the details.
\smallskip

Recall the definition of space $\AMV^p$ in~\eqref{eq:amvdomain} and~\eqref{eq:amvnorm}. 
\begin{prop}\label{prop:improve}
	Let $\Omega\subset X$ be a domain, and $u\in \AMV^p(\Omega)$. Then $u\in M_{loc}^{1/2,p}(\Omega)$. Moreover, if $u\in M_{loc}^{\alpha,p}(\Omega)$ for some $\alpha\in(0,1)$, then $u\in M^{\alpha',p}_{loc}(\Omega)$, where 
	\[
	\alpha'=\frac{2-1/p}{3-\alpha-1/p}>\alpha.
	\]
\end{prop}
\begin{proof}
	Let $K\subset \Omega$ be compact, and define
	\[
	R_K=\frac 16\min\{\dist(K,X\setminus\Omega)^2,r_K^2,1\},
	\]
	where $r_K$ is given by the locally doubling condition \eqref{eq:doubl}. Let $x,y\in K$ satisfy $d(x,y)<R_K$. For any $r\in (d(x,y),r_K)$ we have that
	\begin{align}\label{eq:elest}
	\begin{split}
	|u(x)-u(y)|\le & |u(x)-A^ru(x)|+|u(y)-A^ru(y)|+|A^ru(x)-A^ru(y)|\\
	\le & 2 r \int_0^r(|\Delta_tu(x)|+|\Delta_t u(y)|)\ud t +\frac{Cd(x,y)}{r}\vint_{B_{3r}(x)}|u-u_{B_{3r}(x)}|\ud\mu,
	\end{split}
	\end{align}
	where the third term is estimated using Proposition \ref{lipaprox} with $c=u_{B_{3r}(x)}$ and the first two terms follow from the estimate 
	\begin{align*}
	\abs{u(z) - A^r u(z)} & \leq \frac{2}{r} \int_{ \frac{r}{2}}^r \abs{u(z)-\vint_{B_t(z)} u(w) \ud \mu(w) } \ud t = \frac 2r \int_{ \frac{r}{2}}^r t^2 \abs{\Delta_t u(z)} \ud t \\
	& \leq 2r \int_{ \frac{r}{2}}^r \abs{\Delta_t u(z)} \ud t \leq 2r \int_0^r \abs{\Delta_t u(z)} \ud t,\quad z=x,y.
	\end{align*}
	By choosing $r=d(x,y)^{1/2}$ we have that $d(x,y)<r<r_K$ and thus, by \eqref{eq:elest}, we obtain
	\[
	|u(x)-u(y)|\le d(x,y)^{1/2}[g(x)+g(y)],
	\]
	where
	\[
	g(x)=2 \int_0^{r_K}|\Delta_tu(x)|\ud t+C \M_{3r_K}^\#u(x).
	\]
	This proves the first claim. 
	
Suppose $u\in M^{\alpha,p}_{loc}(\Omega)$, and let $g_K$ be the Haj\l asz gradient and $\tilde{r}_K$ the radius in Definition \ref{def:lochaj}. Set
	\[
	R_K=\frac 16\min\{ \tilde{r}_K,\dist(K,X\setminus\Omega)^{3-\alpha-1/p}, r_K^2\}.
	\]
	From \eqref{eq:elest} we obtain that as long as $6r<\tilde{r}_K$ and $d(x,y) < R_K$, then
	\begin{align}\label{eq:elest2}
	|u(x)-u(y)|\le C r^{2-1/p} \left(\int_0^r(|\Delta_tu(x)|^p+|\Delta_t u(y)|^p)\ud t\right)^{1/p}+C\frac{d(x,y)}{r^{1-\alpha}}[A_{6r}g_K(x)+A_{6r}g_K(y)].
	\end{align}
	Indeed, applying the H\"older inequality to the first term in the right-hand side of \eqref{eq:elest} we obtain, up to a multiplicative constant, that
	\[ r \int_0^r \abs{\Delta_t u(x) } \ud t \leq r \left( \int_0^r \abs{\Delta_t u(x)}^p \ud t\right)^{\frac 1p} r^{\frac{p-1}{p}} = r^{2-\frac 1p} \left( \int_0^r \abs{\Delta_t u(x)}^p\ud t \right)^{\frac 1p}.\]
	The second term in \eqref{eq:elest} is estimated using the Haj\l asz inequality, up to a multiplicative constant, in the following way
	\begin{align*}
	\frac{d(x,y)}{r}\vint_{B_{3r}}|u-u_{B_{3r}(x)}|\ud\mu &= \frac{d(x,y)}{r} \vint_{B_{3r}(x)} \abs{u(w) - \vint_{B_{3r}(x)} u(z) \ud \mu(z)} \ud \mu(w) \\
	& \leq \frac{d(x,y)}{r} \vint_{B_{3r}(x)} \vint_{B_{3r}(x)} \abs{u(w) - u(z)} \ud \mu(z) \ud \mu(w) \\
	& \leq \frac{d(x,y)}{r} \vint_{B_{3r}(x)} \vint_{B_{3r}(x)} d(w,z)^\alpha (g_K(w)+g_K(z)) \ud \mu(z) \ud \mu(w) \\
	& \leq \frac{d(x,y)}{r} (6r)^\alpha \vint_{B_{3r}(x)} \vint_{B_{3r}(x)}  (g_K(w)+g_K(z)) \ud \mu(z) \ud \mu(w) \\
	& \leq C \frac{d(x,y)}{r^{1-\alpha}} [A_{6r} g_K(x) + A_{6r}g_K(y)] .
	\end{align*}
	We choose $r$ in \eqref{eq:elest2} such that
	\begin{align*}
	r^{2-1/p}=\frac{d(x,y)}{r^{1-\alpha}},
	\end{align*}
	i.e.
	\[
	r=d(x,y)^{1/(3-\alpha-1/p)}<r_K^{1/(3-\alpha-1/p)}< \tilde{r}_K,
	\]
	to obtain
	\begin{align*}
	|u(x)-u(y)|\le d(x,y)^{\alpha'}[g(x)+g(y)]
	\end{align*}
	where
	\[
	\alpha'=\frac{2-1/p}{3-\alpha-1/p}
	\]
	and
	\[
	g(x)=C\left(\int_0^{\tilde{r}_K}|\Delta_tu(x)|^p\ud t\right)^{1/p}+C\M_{\tilde{r}_K}g_K(x).
	\]
\end{proof}
Theorem \ref{thm:frachaj} follows immediately from Proposition \ref{prop:improve}.
\begin{proof}[Proof of Theorem \ref{thm:frachaj}]
	Define $\alpha_0=1/2$ and
	\[
	\alpha_{n+1}=\frac{2-1/p}{3-\alpha_n-1/p},\quad n\ge 0.
	\]
	We see that $\alpha_n$ forms an increasing sequence converging to 1. By Proposition \ref{prop:improve} we have that $u\in M^{\alpha_k,p}_{loc}(\Omega)$ for every $k$. The claim follows.
\end{proof}

In light of the fractional Morrey embedding (Proposition \ref{prop:fracmorrey}), Theorem \ref{thm:frachaj} immediately yields the following corollary, which also contains Theorem \ref{thm:holder}.
\begin{corol}\label{cor:holder}
Let $\Omega\subset X$ be a domain in a doubling metric measure space with doubling exponent $Q$. If $p>Q$ and $u\in \AMV_{loc}^p(\Omega)$, then $u$ is locally $\beta$-H\"older continuous for every $\beta<1-Q/p$.

In particular a strongly amv-harmonic function $u$ on $\Omega$ is $\beta$-H\"older continuous for any $\beta<1$. 
\end{corol}
\begin{proof}
	Given $\beta<1-Q/p$, choose $\alpha<1$ which satisfies $\alpha p>Q$ and $\beta<\alpha- Q/p$. By Proposition \ref{prop:fracmorrey} $u$ is locally $(\alpha-Q/p)$-H\"older continuous, and thus locally $\beta$-H\"older continuous.
	
	If $u$ is strongly amv-harmonic on $\Omega$, then in particular $u\in \AMV_{loc}^p(\Omega)$ for every $p<\infty$. This finishes the proof.
\end{proof}

\section{Blow-ups of Haj\l asz--Sobolev functions with finite AMV-norm}\label{sec:blow-up}

To study blow-ups of functions with finite amv-norm, we first review pointed measured Gromov--Hausdorff convergence (pmGH-convergence) of spaces and functions. In the literature there are several variants of pmGH-convergence. Here we follow the presentation of \cite[Section 2]{keith03}, and refer the interested reader to \cite{keith03, HKST, SylvesterNages} for more discussion and the relationship between the various notions. As for the blow-ups there is a plethora of literature, see for instance~\cite{cheeg, he07, keith, kle-mack} and references therein.


\subsection{Pointed measured Gromov--Hausdorff convergence}
Let $F_m\subset Z$ for $m=1,2,\ldots$ and $F\subset Z$ be closed sets in a metric space $Z$. We say that a sequence $(F_m)$ \emph{Hausdorff-converges to} $F$, denoted $F_m\to F$, if
\begin{align*}
\lim_{m\to \infty}\sup_{z\in F_m\cap B_R(q)}\dist_Z(z,F)=0\ \textrm{ and }\
\lim_{m\to \infty}\sup_{z\in F\cap B_R(q)}\dist_Z(z,F_m)=0
\end{align*}
for every $q\in Z$ and $R>0$. If $\nu_m$, for $m=1,2,\ldots$ and $\nu$ are Radon measures on $Z$, we say that a sequence $(\nu_m)$ converges to $\nu$ weakly, denoted $\nu_m\rightharpoonup \nu$, if
\[
\lim_{m\to\infty}\int_Z\varphi \ud\nu_m=\int_Z\varphi \ud\nu
\]
for every continuous function $\varphi:Z\to \R$ with bounded support.
\begin{lemma}\label{lem:Hconv}
Let $(\nu_m)$ be a sequence of measures on $Z$ converging weakly to $\nu$, and suppose $(F_m)$ is a sequence of compact sets Hausdorff-converging to a compact set $F$. Then 
\[
\limsup_{m\to\infty}\nu_m(F_m)\le \nu(F)\le \liminf_{m\to\infty}\nu_m(N_\varepsilon(F_m))
\]
for any $\varepsilon>0$.
\end{lemma}
\begin{proof}
Given $\varepsilon>0$ there exists $m_0$ such that $F\subset N_\varepsilon(F_m)$ and $F_m\subset N_\varepsilon(F)$, whenever $m\ge m_0$. By \cite[Remark 11.4.1]{HKST} we obtain
\begin{align*}
\limsup_{m\to\infty}\nu_m(F_m)\le \limsup_{m\to\infty}\nu_m(\bar N_\varepsilon(F))\le \nu(\bar N_\varepsilon(F)).
\end{align*}
Taking infimum over $\varepsilon>0$ we  obtain the left-hand side inequality in the assertion. Similarly, to see the second inequality we use \cite[Remark 11.4.1]{HKST} and estimate
\begin{align*}
\nu(F)\le \nu(N_{\varepsilon}(F))\le \liminf_{m\to\infty}\nu_m(N_{\varepsilon}(F))\le \liminf_{m\to\infty}\nu_m(N_{2\varepsilon}(F_m)).
\end{align*}
\end{proof}
 
For the next definition, we recall that a pointed metric measure space $X=(X,d,\mu,p)$ consists of a metric measure space $(X,d,\mu)$ and a distinguished point $p\in X$. We consider only proper spaces here.
\begin{definition}\label{def:pmGH}
A sequence $(X_m)=(X_m,d_m,\mu_m,p_m)$ of pointed proper metric measure spaces \emph{pmGH-converges} to a pointed proper metric measure space $X=(X,d,\mu,p)$, denoted $X_m\stackrel{pmGH}{\longrightarrow} X$, if there exists a pointed proper metric space $(Z,q)$ and isometric embeddings $\iota_m:X_m\to Z$ for $m=1,2,\ldots$ and $\iota:X\to Z$ so that
\begin{itemize}
	\item[(1)] $\iota_m(p_m)=\iota(p)=q$ for all $m$, and $\iota_m(X_m)$ Hausdorff-converges to $\iota(X)$;
	\item[(2)] the measures $\iota_{m*}\mu_m$ weakly converge to $\iota_*\mu$, as $m\to
	\infty$. 
\end{itemize}
\end{definition}
We also define Gromov--Hausdorff convergence for sequences of functions. Note that in Definition \ref{def:fGH}, the measures play no role. Since we consider pointed measured spaces, we nevertheless include them (see also Definition \ref{def:wfGH}). 
\begin{definition}\label{def:fGH}
	Let $u_m:X_m\to \R$ for $m=1,2,\ldots$ and $u:X\to \R$ be continuous functions on pointed proper metric measure spaces. We say that the sequence $(u_m)$ \emph{Gromov--Hausdorff converges to} $u$, denoted $u_m\stackrel{GH}{\longrightarrow} u$,	if there are isometric embeddings $\iota_m:X_m\to Z$ $m=1,2,\ldots$ and $\iota:X\to Z$ satisfying (1) and (2) in Definition \ref{def:pmGH} and
	\begin{itemize}
		\item[(3)] $u_m(z_m)\to u(z)$ whenever $z_m\in X_m$, $z\in X$, and $\iota_m(z_m)\to \iota(z)$, as $m\to \infty$.
	\end{itemize}
\end{definition}

\bigskip\noindent The embeddings $\iota_m:X_m\to Z$ and $\iota:X\to Z$ satisfying (1) and (2) (resp. (3)) in Definition \ref{def:pmGH} are said to \emph{realize} the convergence $X_m\stackrel{pmGH}{\longrightarrow} X$ (resp. $u_m\stackrel{GH}{\longrightarrow} u$).

Two central properties of Gromov--Hausdorff convergence are its compactness properties (see Proposition \ref{prop:GHcpt} below) and the stability of many important properties in metric geometry and analysis under the Gromov--Hausdorff convergence. For our purposes, the stability of length spaces and the doubling property of the measure is important; for a detailed discussion see \cite[Section 11]{HKST} and \cite{keith03}. 

The following compactness property is proved in \cite[Proposition 3]{keith03}, see also \cite{HKST}. For a pointed metric space $(X,d,p)$ and $\varepsilon, R>0$, let $N_X(\varepsilon,R)$ denote the maximal number of disjoint closed balls of radius $\varepsilon$ inside a ball $B_R(p)$. We say that a sequence $(X_m,d_m,p_m)$ of pointed metric spaces is \emph{totally bounded} if
\[
\sup_m N_{X_m}(\varepsilon,R)<\infty
\]
for every $\varepsilon,R>0$.

\begin{prop}[cf. Proposition 3, \cite{keith}]\label{prop:GHcpt}
Let $(X_m)=(X_m,d_m,\mu_m,p_m)$ be a totally bounded sequence of proper metric measure spaces, satisfying
\begin{align}\label{eq:totbound}
\sup_m\mu_m(B_R(p_m))<\infty \quad \textrm{for every }R>0.
\end{align}
Then there exists a subsequence and a pointed proper metric measure space $X=(X,d,\mu,p)$ so that $X_m\stackrel{pmGH}{\longrightarrow} X$. 

\end{prop}

There is also a compactness result for sequences of functions on pointed spaces, which can be proved using Proposition \ref{prop:GHcpt} and a diagonal argument, as in the proof of the Arzela--Ascoli theorem. Under a different notion of convergence (which is equivalent to ours under the hypotheses there), Proposition \ref{prop:fGHcpt} appears in \cite{keith}. For the additional statement (2), see \cite[Section 11.5]{HKST} and \cite{keith03}.

To state the result, let $(X_m)=(X_m,d_m,\mu_m,p_m)$ be a sequence of pointed proper metric measure spaces. We say that a sequence $(f_m)$ of functions $f_m:X_m\to \R$ for $m=1,2,\ldots$ is \emph{equicontinuous} if, for every $\varepsilon,R>0$, there exists $\delta>0$ such that if $x_m,y_m\in B_R(p_m)$ satisfy $d(x_m,y_m)<\delta$, then $|f_m(x_m)-f_m(y_m)|<\varepsilon$ uniformly in $m$.
\begin{prop}\label{prop:fGHcpt}
	Let $(X_m)=(X_m,d_m,\mu_m,p_m)$ be a totally bounded sequence of pointed proper metric measure spaces satisfying \eqref{eq:totbound}. If a sequence of functions $(f_m)$, where $f_m:X_m\to \R$ for $m=1, 2,\ldots$, is an equicontinuous sequence of functions, for which
	\[
	\sup_m|f_m(p_m)|<\infty,
	\]
	then there exists  a proper pointed metric measure space $X$ and a continuous function $f:X\to \R$ such that, up to extracting a subsequence,  $f_m\stackrel{GH}{\longrightarrow} f$. 
	
	Moreover,
	\begin{itemize}
		\item[(1)] if each $f_m$ is $L$-Lipschitz, then $f$ is $L$-Lipschitz;
		\item[(2)] if each $X_m$ is a doubling length space with doubling constants not exceeding a positive constant $C$, then $X$ is a length space with doubling constant not exceeding $C$.
	\end{itemize}
\end{prop}

To study tangents of Haj\l asz--Sobolev functions, we also consider a notion of weak convergence for functions. The following definition is a slight modification of the weak convergence in \cite{SylvesterNages}.

\begin{definition}\label{def:wfGH}
	Let $(X_m)=(X_m,d_m,\mu_m,p_m)$ and $X=(X,d,\mu,p)$ be pointed proper metric measure spaces. A sequence $(u_m)$ of Borel functions $u_m\in L^1_{loc}(X_m)$ converges \emph{weakly} to a Borel function $u\in L^1_{loc}(X)$, denoted $u_m\stackrel{GH}{\rightharpoonup} u$, if there exist isometric embeddings $\iota_m:X_m\to Z$, $\iota:X\to Z$ satisfying (1) and (2) in Definition \ref{def:pmGH}, and for which
	\begin{itemize}
		\item[(3')] $\iota_{m*}((u_m)_+\ud\mu_m)\rightharpoonup \iota_*(u_+\ud\mu)$ and $\iota_{m*}((u_m)_-\ud\mu_m)\rightharpoonup \iota_*(u_-\ud\mu)$, as $m\to \infty$.
	\end{itemize}
Here, for a function $f:Z\to \R$ we denote
\[
f_+=\max\{ f,0 \}\textrm{ and } f_-=-\min\{f,0 \}.
\]
\end{definition}

The Gromov-Hausdorff convergence of functions is analogous to the uniform convergence on compact sets and indeed coincides with this notion if $X_m=X=Z$ for all $m$. The weak convergence of functions as in Definition~\ref{def:wfGH} corresponds to weak convergence of signed measures. In keeping with these analogies, we indeed have the natural implication between the two notions.

\begin{lemma}\label{lem:uniftoweak}
	Let metric measure spaces $(X_m)$ and $X$ be as in Definition~\ref{def:wfGH}. Suppose further, that a sequence $(u_m)$ of Borel functions $u_m:X_m \to \R$ Gromov--Hausdorff converges to a continuous function $u:X\to \R$. Then $(u_m)$ converges to $u$ weakly.
	
	More precisely, if $\iota_m:X_m\to Z$, $\iota:X\to Z$ realize the convergence $u_m\stackrel{GH}{\longrightarrow} u$, then \[
	\lim_{m\to \infty}\int_{X_m}(\varphi\circ\iota_m) (u_m)_\pm \ud\mu_m=\int_X(\varphi\circ\iota) u_\pm\ \ud\mu.
	\]
	for any boundedly supported continuous $\varphi:Z\to \R$.
\end{lemma}
\begin{proof}
	It is easy to see that if $u_m\stackrel{GH}{\longrightarrow} u$, then $(u_m)_\pm\stackrel{GH}{\longrightarrow} u_\pm$, and the embeddings realizing the first convergence also realize the latter convergence. Thus we may assume that $u_m$ is non-negative for all $m$, and the embeddings $\iota_m:X_m\to Z,\ \iota:X\to Z$ realize the convergence $u_m\stackrel{GH}{\longrightarrow} u$. It suffices to show that $\iota_{m*}(u_m\ud\mu)\rightharpoonup \iota_*(u\ud\mu)$.
	
Let $\tilde u:Z\to \R$ be a continuous extension of $u\circ\iota^{-1}:\iota(X)\to \R$, and set 
$$
 \tilde u_m= u_m\circ\iota_m^{-1}|_{\iota_m(X_m)}.
$$
 Given any continuous $\varphi:Z\to \R$ with bounded (thus compact) support, we have
	\begin{align*}
	\int_{X_m}(\varphi\circ\iota_m) u_m \ud\mu_m-\int_X(\varphi\circ\iota) u\ \ud\mu=\int_Z\varphi\ (\tilde u_m-\tilde u) \ \iota_{m\ast}(\ud\mu_m)+\int_Z\varphi\ \tilde u\ \iota_{m\ast}(\ud\mu_m)-\int_Z\varphi\ \tilde u\ \iota_*(\ud\mu).
	\end{align*}
	Since $\iota_{m*}(\ud\mu_m)\rightharpoonup \iota_*(\ud\mu)$, it suffices to prove that
	\begin{align*}
	\int_Z\varphi\ (\tilde u_m-\tilde u) \ \iota_{m\ast}(\ud\mu_m)\to 0.
	\end{align*}
	If $B\subset Z$ is a closed ball containing the support of $\varphi$, we obtain
	\begin{align*}
	\limsup_{m\to\infty}\left|\int_Z\varphi\ (\tilde u_m-\tilde u) \ \iota_{m\ast}(\ud\mu_m)\right|\le &\limsup_{m\to\infty} \|\varphi\|_{L^\infty(B)}\mu_m(\iota_m^{-1}B)\|\tilde u_m-\tilde u\|_{L^{\infty}(\iota_m(X_m)\cap B)}\\
	\le & \|\varphi\|_{L^{\infty}(B)}\mu(\iota^{-1}B)\limsup_{m\to\infty}\|\tilde u_m-\tilde u\|_{L^{\infty}(\iota_m(X_m)\cap B)}.
	\end{align*}
	Suppose $\limsup_{m\to\infty}\|\tilde u_m-\tilde u\|_{L^{\infty}(\iota_m(X_m)\cap B)}> \varepsilon_0$ for some $\varepsilon_0>0$. Then, there is a sequence $(x_m)$ such that $x_m\in X_m$
	with the corresponding sequence $(z_m)$ defined $z_m:= \iota_m(x_m)\in B$ and
	\[
	|\tilde u_m(z_m)-\tilde u(z_m)|=|u_m(x_m)-\tilde u(\iota_m(x_m))|>\varepsilon_0
	\]
	for $m$ large enough. Since $B$ is compact, a subsequence of $(z_m)$, denoted the same for simplicity, satisfies $z_m\to z\in \iota(X)$ for some $z$, as $m\to \infty$. By the Gromov--Haudorff convergence $u_m\stackrel{GH}{\longrightarrow} u$ and the continuity of $\tilde u$, we obtain
	\begin{align*}
	\lim_{m\to\infty}|u_m(x_m)-\tilde u(\iota_m(x_m))|=|u(\iota^{-1}(z))-\tilde u(z)|=0, 
	\end{align*}
	which is a contradiction. This completes the proof.
\end{proof}

\subsection{Gromov--Hausdorff convergence and averaging operators}

In this section we show the stability of operators $A_r$ and $A_r^*$ under the Gromov-Hausdorff convergence.
 
\begin{prop}\label{prop:intconv}
Let $(X_m)=(X_m,d_m,\mu_m,p_m)$ and $X=(X,d,\mu,p)$ be proper locally doubling length spaces. Suppose the sequence $(u_m)$ of Borel functions $u_m\in L_{loc}^1(X_m)$ converges weakly to a Borel function $u\in L^1_{loc}(X)$, and let $\iota_m:X_m\to Z$ for $m=1,2,\ldots$ and $\iota:X\to Z$ realize this convergence. If $z_m\in X_m$ for $m=1,2,\ldots$ and $z\in X$ are such that $\iota_m(z_m)\to \iota(z)$, as $m\to \infty$, then
\[
\lim_{m\to\infty}\int_{B_r(z_m)}u_m\ud\mu_m=\int_{B_r(z)}u\ud\mu
\]
for any $r>0$.
\end{prop}
\begin{proof}
By considering functions $u_{m\pm}$ for $m\geq 1$, we may assume without the loss of generality that all the functions $u_m$ and $u$ are non-negative. We note that
\begin{align}\label{eq:ballconv}
\iota_m(B_r(z_m))\to \iota(B_r(z)),\quad m\to \infty
\end{align}
for every $r>0$ in the sense of Hausdorff-convergence. Indeed, given $q\in Z$, $R>0$ and arbitrary $\varepsilon>0$ the convergence $\iota_m(X_m)\to \iota(X)$ and $\iota_m(z_m)\to \iota(z)$ for $m\to \infty$ imply that, for large enough $m$
\begin{align*}
B_R(q)\cap \iota_m(B_r(z_m))=B_R(p)\cap \iota_m(X_m)\cap B_r(\iota_m(z_m))\subset N_\varepsilon(\iota(X))\cap B_{r+\varepsilon}(\iota(z))=N_\varepsilon(\iota(B_r(z)))
\end{align*}
and 
\begin{align*}
B_R(q)\cap \iota(B_r(z))=B_R(p)\cap \iota(X)\cap B_r(\iota(z)))\subset N_\varepsilon(\iota_m(X_m))\cap B_{r+\varepsilon}(\iota_m(z_m))=N_\varepsilon(\iota(B_r(z_m))).
\end{align*}
 The convergence \eqref{eq:ballconv} follows. Note that the property of being a length space was used in the last equality, cf. \cite[Lemma 11.3.10]{HKST}. Denote $\ud\nu=(u\circ\iota^{-1})\ud(\iota_*\mu)$ and $\nu_m=(u_m\circ\iota^{-1})\ud(\iota_{m*}\mu_m)$. By assumption, $\nu_m$ converges weakly to $\nu$ and, by Lemma \ref{lem:Hconv}, we have 
\begin{align*}
\limsup_{m\to\infty}\nu_m(\bar B_r(z_m))\le \nu(\bar B_r(z))\le \liminf_{m\to\infty}\nu_m(B_{\varepsilon+r}(z_m))
\end{align*}
for any $\varepsilon>0$. Recall from Section~\ref{subs:loc-lip} notation $A_{r,R}(x)=\bar B_R(x)\setminus B_r(x)$ for $x\in Z$ and $r\leq R$. It holds that
\begin{align*}
\liminf_{m\to\infty}\nu_m(B_{\varepsilon+r}(z_m))&=\liminf_{m\to\infty}[\nu_m(\bar B_{r}(z_m))+\nu_m(A_{r,r+\varepsilon}(z_m))]\\
&\le \liminf_{m\to\infty}\nu_m(B_{r}(z_m))+\limsup_{m\to\infty}\nu_m(A_{r,r+\varepsilon}(z_m)).
\end{align*}
The argument used to establish \eqref{eq:ballconv} also yields that
\[
A_{r,r+\varepsilon}(z_m)\to A_{r,r+\varepsilon}(z)
\]
in the sense of Hausdorff-convergence. Applying Lemma \ref{lem:Hconv} once more we obtain
\[
\limsup_{m\to\infty}\nu_m(\bar B_r(z_m))\le \nu(\bar B_r(z))\le \liminf_{m\to\infty}\nu_m(B_{r}(z_m))+\nu(A_{r,r+\varepsilon}(z)).
\]
Since $\mu$ is a locally doubling measure on a length space, it has an annular decay property (see the discussion after Definition~\ref{defn-an-dec}). Therefore, $\lim_{\varepsilon\to 0}\mu(A_{r,r+\varepsilon}(z))=0$ which, by the absolute continuity of $\nu$ with respect to $\mu$, implies $\lim_{\varepsilon\to 0}\nu(A_{r,r+\varepsilon}(z))=0$. We have obtained 
\[
\limsup_{m\to\infty}\nu_m(\bar B_r(z_m))\le \nu(\bar B_r(z))\le \liminf_{m\to\infty}\nu_m(\bar B_{r}(z_m)),
\]
which completes the proof.
\end{proof}

Proposition \ref{prop:intconv} has the following immediate corollary.

\begin{corol}\label{cor:rdeltaconv}
Let $(X_m)=(X_m,d_m,\mu_m,p_m)$ and $X=(X,d,\mu,p)$ be proper locally doubling length spaces, and $(u_m)$ be a sequence of Borel functions $u_m\in L_{loc}^1(X_m)$ converging weakly to a Borel function $u\in L^1_{loc}(X)$. Then
\begin{itemize}
	\item[(a)] $A_r^{\mu_m}u_m\stackrel{GH}{\longrightarrow} A_r^\mu u$, and 
	\item[(b)] $(A_r^{\mu_m})^*u_m\stackrel{GH}{\longrightarrow} (A_r^{\mu})^* u$
\end{itemize}
for each $r>0$. In particular, if $u$ is continuous, then 
\[
\Delta_r^{\mu_m}u_m\stackrel{GH}{\rightharpoonup}\Delta_r^\mu u\quad \textrm{and}\quad (\Delta_r^{\mu_m})^*u_m\stackrel{GH}{\rightharpoonup}(\Delta_r^{\mu})^*u.
\]

\end{corol}
\begin{proof}
The first claim follows directly from Proposition \ref{prop:intconv}. Note that, if $\iota_m:X_m\to Z$ and $\iota:X\to Z$ realize the convergence then $X_m\stackrel{pmGH}{\longrightarrow} X$, Proposition \ref{prop:intconv} implies, in particular, that
\[
\mu_m(B_r(z_m))\to \mu(B_r(z))\quad \textrm{ whenever } \iota_m(z_m)\to \iota(z).
\]
It follows that, for any $r>0$, the sequence $(f_m)$ of functions $f_m\in L_{loc}^1(X_m)$,
\[
f_m(z)=\frac{u_m(z)}{\mu_m(B_r(z))}
\]
weakly converges to
\[
f(z):=\frac{u(z)}{\mu(B_r(z))}.
\]
This yields part (b) of the assertion.
\end{proof}

\subsection{Blow-ups of Sobolev functions with finite amv-norm}\label{subs:blowups}

Consider a proper locally doubling length space $X=(X,d,\mu)$. Given a point $x\in X$ and $r>0$, the pointed metric measure space
\[
X_r=(X,d_r,\mu_r,x),\quad d_r:=\frac{d}{r},\quad \mu_r:=\frac{1}{\mu(B(x,r))}\mu
\]
is called a \emph{rescaling} of $X$ (at $x$ by $r$).

Let $(r_m)$ be a sequence of positive numbers converging to zero, and denote $X_m:=X_{r_m}$. A pointed measured Gromov--Hausdorff limit $X_\infty$ of $X_m$ is called a \emph{tangent space of $X$ at $x$ subordinate to $(r_m)$.}

Similarly, if $f:X\to \R$ is a function, $x\in X$ and $r>0$, then a function
\[
f_r:=\frac{f-f(x)}{r}:X_r\to \R
\]
is a rescaling of $f$ at $x$ by $r$. A Gromov--Hausdorff limit $f_\infty:X_\infty\to \R$ of a sequence $(f_m)$, where $f_m:=f_{r_m}$ for $m=1,2,\ldots$ is called a \emph{tangent of $f$ at $x$ subordinate to $(r_m)$.} If the convergence $f_m\to f_\infty$ as $m\to \infty$ is weak (cf. Definition \ref{def:wfGH}), then we say that $f_\infty$ is an \emph{approximate tangent of $f$ at $x$, subordinate to $(r_m)$.}

It is worth remarking that, in general, tangents are highly non-unique -- different sequences can produce different limits. However, any sequence of rescalings is totally bounded and satifies \eqref{eq:totbound}. Moreover, $f_m(x)=0$ for all $m$, and thus Proposition \ref{prop:fGHcpt} implies the existence of tangents of Lipschitz functions at any point.

\begin{prop}\label{prop:liptan}
	Let $X$ be a proper locally doubling metric measure space and $f:X\to \R$ an $L$-Lipschitz function. Fix a sequence $(r_m)$ of positive numbers converging to zero. Then, for any $x\in X$, there exists a tangent space $X_\infty$ at $x$ subordinate to $(r_m)$ and a tangent function $f_\infty: X_\infty \to \R$ such that, up to extracting a subsequence of rescalings,  $f_m\stackrel{GH}{\longrightarrow} f_\infty$.
	\end{prop}
In particular, any tangent space $X_\infty$ is doubling, and any tangent $f_\infty:X_\infty\to\R$ is $L$-Lipschitz. Next, we present a variant for Haj\l asz--Sobolev functions. 
\begin{prop}\label{prop:hajtan}
Let $X$ be a proper locally doubling metric measure space, $p>1$, and $u\in M^{1,p}(X)$. Given a sequence $(r_m)$, it holds that for $\mu$-almost every point $x\in X$ there exist a tangent space $X_\infty$ at $x$ subordinate to $(r_m)$, a subsequence of the rescalings $u_m:X_m\to\R$ and a Lipschitz function $u_\infty:X_\infty\to\R$, such that, up to extracting a subsequence of rescalings, $u_m\stackrel{GH}{\rightharpoonup} u_\infty$.
\end{prop}
\begin{proof}
	Let $g\in L^p(X)$ be a Haj\l asz upper gradient of $u$,  and set
	\[
	E_n=\{x\in X: g(x)>n \},\quad n\in\N.
	\]
	Then
	\[
	\lim_{n\to\infty}\mu(E_n)=0\quad\textrm{and}\quad \lim_{n\to\infty}\int_{E_n}g\ud\mu=0.
	\]
	Hence, there exists a null set $N\subset X$ for which every $x\in X\setminus N$ has the property that 
	\begin{align}\label{eq:dens}
	\lim_{r\to 0}\frac{\mu(B_r(x)\cap E_n)}{\mu(B_r(x))}=0\quad\textrm{ and }\quad\lim_{r\to 0}\frac{1}{\mu(B_r(x))}\int_{B_r(x)\cap E_n} g(y)\ud\mu(y)=0
	\end{align}
	for some $n\in \N$. We fix $x\in X\setminus N$ and $n\in\N$ satisfying \eqref{eq:dens}.
	
	Note that $u|_{X\setminus E_n}$ is $2n$-Lipschitz and let $\tilde u_n:X\to\R$ be a $2n$-Lipschitz extension of $u|_{X\setminus E_n}$. By Proposition \ref{prop:liptan} there is a subsequence of the rescalings $(\tilde u_n)_m:X_m\to\R$ and a Lipschitz function $\tilde u:X_\infty\to\R$ so that $(\tilde u_n)_m\stackrel{GH}{\longrightarrow} \tilde u$.
	
	
	
	We show that, for this subsequence, the rescalings $u_m:X_m\to \R$ converge weakly to $\tilde u$. (This is different from claiming that $(\tilde u_n)_m\stackrel{GH}{\rightharpoonup} \tilde u$, which follows from Lemma \ref{lem:uniftoweak}.)
	
	Let $Z$ be a proper metric space and $\iota_m:X_m\to Z$ isometric embeddings realizing the convergence $(\tilde u_n)_m\to \tilde u$. Given $\varphi\in C_b(Z)$, fix a large number $R>0$ so that $\spt \varphi\subset B_R(\iota_m(x))$ for all $m\in\N$. Then $\spt(\varphi\circ\iota_m)\subset B^{X_m}_R(x)=B_{r_mR}(x)$. We have
	\begin{align*}
	&\left|\int_{X_m}\varphi\circ\iota_m\ u_m\ud\mu_m-\int_{X_\infty}\varphi\circ\iota_\infty\ \tilde u\ud\mu_\infty\right|\\
	= & \left|\int_{X}\varphi\circ\iota_m\ (\tilde u_n)_m\ud\mu_m-\int_{X_\infty}\varphi\circ\iota_\infty\ \tilde u\ud\mu_\infty+\int_{E_n}\varphi\circ\iota_m\ [u_m-(\tilde u_n)_m]\ud\mu_m\right|\\
	\le & \left|\int_{X}\varphi\circ\iota_m\ (\tilde u_n)_m\ud\mu_m-\int_{X_\infty}\varphi\circ\iota_\infty\ \tilde u\ud\mu_\infty\right|+ \int_{E_n}|\varphi\circ\iota_m\ u_m|\ud\mu_m+\int_{E_n}|\varphi\circ\iota_m\ (\tilde u_n)_m|\ud\mu_m
	\end{align*}
	The first term converges to zero since $(\tilde u_n)_m\stackrel{GH}{\longrightarrow}\tilde u$. We may estimate the second term by
	\begin{align*}
	\int_{E_n}|\varphi\circ\iota_m|\ |u_m|\ud\mu_m\le \frac{\|\varphi\|_\infty}{\mu(B_{r_m}(x))}\int_{E_n\cap B_{r_mR}(x)}R(g(x)+g(y))\ud\mu(y)
	\end{align*}
	which, by \eqref{eq:dens} converges to zero as $m\to\infty$. Similarly, 
	\begin{align*}
	\int_{E_n}|\varphi\circ\iota_m|\ |(\tilde u_n)_m|\ud\mu_m\le 2nR\|\varphi\|_\infty\frac{\mu(E_n\cap B_{r_mR}(x))}{\mu(B_{r_m}(x))}
	\end{align*}
	converges to zero as $m\to\infty$.	This completes the proof.
\end{proof}

By a suitable cut-off argument, we obtain the following corollary whose proof we omit.
\begin{corol}\label{cor:hajtan}
	Let $\Omega\subset X$ be a domain in a proper locally doubling metric measure space, and let $u\in M^{1,p}_{loc}(\Omega)$. Given a sequence $r_m\downarrow 0$, it holds that for $\mu$-almost every $x\in \Omega$,  there exist a tangent space $X_\infty$ at $x$ subordinate to $(r_m)$, a subsequence of the rescalings $u_m:X_m\to\R$ and a Lipschitz function $u_\infty:X_\infty\to\R$, such that, up to extracting a subsequence of rescalings, $u_m\stackrel{GH}{\rightharpoonup} u_\infty$.
\end{corol}

We are now ready to prove that having finite amv-norm forces tangent maps to be mean value harmonic.

\begin{proof}[Proof of Theorem \ref{thm:blowup}]
	
	By Corollary~\ref{cor:hajtan}, $u$ has approximate tangent maps, subordinate to $(r_m)$, for $\mu$-almost every $x\in X$. For any $k\in\N$, $r\in \mathbb Q_+$, and compact $K\subset \Omega$, we have
	\[
	\int_K\liminf_{m\to\infty}\vint_{B(x,kr_m)}|\Delta_{r_mr}u|d\mu \,d\mu(x)\le \liminf_{m\to\infty}\int_K\vint_{B(x,kr_m)}|\Delta_{r_mr}u|d\mu\ d\mu(x)\le C\limsup_{\rho\to 0}\int_K|\Delta_\rho u|d\mu,
	\]
	cf. \cite[Theorem 3.5]{Aldaz18}. Thus, almost every $x_0\in X$ has the following property: for every  subsequence of $(r_m)$ there exists a further subsequence (not relabeled), for which
	\begin{equation}\label{eq:Damv}
	\lim_{m\to\infty}\vint_{B(x_0,kr_m)}|\Delta_{r_mr}u|d\mu<\infty\quad\textrm{for every $k\in\N$ and $r\in \mathbb Q_+$.}
	\end{equation}

	Let $x_0\in \Omega$ be a point where \eqref{eq:Damv} and the claim of Corollary~\ref{cor:hajtan} hold for $(r_m)$. Furthermore, let $(X_\infty,d_\infty,\mu_\infty,x_\infty)$ be a pointed measured Gromov--Hausdorff limit of a sequence
	\[
	X_m=(X,d_m,\mu_m,x_0)=\left(X, \frac{d}{r_n},\frac{\mu}{\mu(B_{r_m}(x_0))},x_0\right),
	\]
	and $u_\infty:X_\infty\to \R$ a weak limit of the sequence
	\[
	u_m:=\frac{u-u(x_0)}{r_m}:X_m\to \R,
	\]
	for a subsequence of $(r_m)$. We pass to a further subsequence (again not relabeled) for which \eqref{eq:Damv} holds. Note that
	\begin{align*}
	A_r^{\mu_m}u_m(z)-u_m(z)=\vint_{B_{r_mr}(z)}\frac{u(y)-u(z)}{r_m}d \mu(y)=r_mr^2\Delta_{r_mr}u(z)
	\end{align*}
	for any $z\in X_m$ and $r>0$. Let $\iota_m:X_m\to Z$ and $\iota:X_\infty\to Z$ realize the convergence $u_m\stackrel{GH}{\rightharpoonup} u_\infty$. Corollary \ref{cor:rdeltaconv} implies that
\begin{align*}
\int_{X_\infty}\varphi\circ\iota_\infty(z)\left[A_r^{\mu_\infty}u_\infty(z)-u_\infty(z)\right]\ud \mu_\infty(z)&=\lim_{m\to\infty}\int_{X_m}\varphi\circ\iota_m(z)\left[A_r^{\mu_m}u_\infty(z)-u_m(z)\right]\ud \mu_m(z)\\
&=\lim_{m\to\infty}r_mr^2\int_{X_m} \varphi\circ\iota_m(z)\Delta_{r_mr}u(z)d\mu_m(z)
\end{align*}
for every compactly supported $\varphi\in C(Z)$. Fix such $\varphi\in C_c(Z)$, and not that for large enough $k$, the support of $\varphi\circ\iota_m$ satisfies $\spt(\varphi\circ\iota_m)\subset B_{kr_m}(x_0)$ for all $m$.

Since
\begin{align*}
\left|\int_{X_m} \varphi\circ\iota_m(z)\Delta_{r_mr}u(z)d\mu_m(z)\right|\le & \frac{\|\varphi\circ\iota_m\|_\infty}{\mu(B_{r_m}(x_0))}\int_{B_{kr_m}(x_0)}|\Delta_{r_mr}u(z)|d\mu(z)\\
\le& C_{k,\varphi} \vint_{B_{kr_m}(x_0)}|\Delta_{r_mr}u|d\mu<\infty
\end{align*}
it follows that 
\begin{align*}
\int_{X_\infty}\varphi\circ\iota_\infty(z)\left[(u_\infty)_{B_\infty(z,r)}-u_\infty(z)\right]d \mu_\infty(z)=0.
\end{align*}
Since $\varphi$ is arbitrary, this establishes the claim for all rational $r>0$. Since $X_\infty$ is a length space and $\mu_\infty$ a doubling measure, \cite[Corollary 2.2]{buc} implies that $(u_\infty)_{B_\infty(z,r_m)}\to (u_\infty)_{B_\infty(z,r)}$ whenever $r_m\to r$. Thus the claim follows for arbitrary $r>0$.
\end{proof}

\section{Weighted Euclidean spaces: elliptic PDEs and amv-harmonic functions}\label{sec:eucl}

In this final section we consider amv-harmonicity in weighted Euclidean domains and prove Theorem \ref{thm:weight} which characterizes weak (and in some cases strong) amv-harmonic functions as weak solutions of the elliptic PDE, cf. \eqref{w-div-oper}:
$$
L_wu=0.
$$

To set up notation, fix a norm $\|\cdot\|$ on $\R^n$ and a domain $\Omega\subset \R^n$. Denote $\Ha^n_\Omega:=\Ha^n|_\Omega$, where the Hausdorff measure is taken with respect to the metric induced by $\|\cdot\|$. Given a positive, locally Lipschitz function $w:\Omega\to (0,\infty)$, we consider the metric measure spaces $$\Omega:=(\Omega,\|\cdot\|,\Ha^n_\Omega)\quad\textrm{and}\quad\Omega_w=(\Omega,\|\cdot\|,w\Ha^n_\Omega).$$
Throughout the section, $B^n$ denotes the (open) unit ball $B_1(0)$ with respect to $\|\cdot\|$. We will use the shorthand notation
\[
\ud y=\ud\Ha^n_\Omega(y),\quad  A_r:=A_r^{\Ha^n_\Omega},\quad\textrm{and}\quad \Delta_r:= \Delta_r^{\Ha^n_\Omega},
\]
as well as
\[
A_r^w:=A_r^{w\Ha^n_\Omega}\quad\textrm{and}\quad \Delta_r^w:=\Delta_r^{w\Ha^n_\Omega}.
\]
The following elementary facts will be used throughout the sequel.
\begin{remark}\label{rmk:locdoubl}
	Since $w$ is continuous and strictly positive on $\Omega$, the measure $w\Ha^n_\Omega$ is locally doubling. Moreover,
	\begin{itemize}
		\item[(1)] $A_r^wf\to f$ locally uniformly in $\Omega$ as $r\to 0$, whenever $f:\Omega\to \R$ is continuous;
		\item[(2)] For each $p\in [1,\infty]$, $L^p_{loc}(\Omega)=L^p_{loc}(\Omega_w)$ as sets, and $L^p$-convergence on compact subsets of $\Omega$ with respect to $\Ha^n_\Omega$ and $w\Ha^n_\Omega$ agree.
	\end{itemize}
\end{remark}

The next lemma provides two different representations for $\Delta_r^w$ in terms of $\Delta_r$ and will prove useful in the next subsection.

\begin{lemma}\label{lem:rep}
	Suppose $f\in L^1_{loc}(\Omega)$. Then
	\begin{equation}\label{eq:2ident}
		\Delta_r^wf=\frac{1}{A_rw}(\Delta_r(fw)-f\Delta_rw),
	\end{equation}
	and
	\begin{equation}\label{eq:wlapl}
		\Delta_r^wf=\Delta_rf+\frac{1}{A_rw}\langle f,w\rangle_r,
	\end{equation}
	where
	\[
	\langle f,g\rangle_r(x):= \vint_{B_r(x)}\frac{f(y)-A_rf(x)}{r}\frac{g(y)-A_rg(x)}{r}\ud y,
	\]
	for  $f,g\in L_{loc}^1(\Omega)$ such that $fg\in L^1_{loc}(\Omega)$, and all balls $B_r(x)\subset \Om$.
\end{lemma}
\begin{proof}
	The first identity is a direct consequence of the pointwise identity
	\[
	(f(y)-f(x))w(y)=f(y)w(y)-f(x)w(x)+f(x)(w(x)-w(y)),\quad x,y\in\Omega.
	\]
	To see the second assertion~\eqref{eq:wlapl}, fix $r>0$ and denote for brevity $\mu:=w\Ha^n_\Omega$. Then 
	\begin{align*}
		\frac{\langle f,w\rangle_r(x)}{A_rw(x)}=&\frac{1}{A_rw(x)}\vint_{B_r(x)}\frac{(f(y)-A_rf(x))(w(y)-A_rw(x))}{r^2}\ud y\\
		=&\frac{1}{A_rw(x)}\vint_{B_r(x)}\frac{f(y)w(y)-A_rw(x)f(y)-A_rf(x)w(y)+A_rf(x)A_rw(x)}{r^2}\ud y\\
		=&\frac{1}{r^2A_rw(x)}\left[\vint_{B_r(x)}fw\ud y-A_rf(x)A_rw(x)\right]=\frac{A_r^wf(x)-A_rf(x)}{r^2}\\
		=&\Delta_r^wf(x)-\Delta_rf(x),\quad x\in \Omega.
	\end{align*}
\end{proof}

\subsection{$\AMV^p_{loc}(\Omega_w)$ as a second order Sobolev space}
To prove Theorem \ref{thm:weight}, we start by analyzing the operator \eqref{w-div-oper}. Note that, for $u\in L^1_{loc}(\Omega)$,
\begin{equation*}
	L_wu:=\frac{1}{2}\div(M\nabla u)+\frac{1}{w}\langle \nabla w,M\nabla u\rangle
\end{equation*}
is a distribution, and is given by a locally integrable function when $u\in W^{2,1}_{loc}(\Omega)$. Recall that the matrix of second moments of $B^n$ is defined as
\[
M=(m_{ij}),\quad m_{ij}=\vint_{B^n} y_iy_j\ud y,\quad i,j=1,\ldots,n.
\]
\begin{remark}\label{rmk:sym}
	The matrix $M$ is symmetric and positive definite, cf.~\cite[Sections 4.1-4.3]{ak}. Since $B^n$ is symmetric we have that
	\[
	\vint_{B_r(x)}(y-x)_i\ud y=r\vint_{B^n}z_i\ud z=0,\quad i=1,\ldots,n
	\]
	for $x\in \R^n$ and $r>0$.
\end{remark}

In this subsection we show that, when $u\in W^{2,p}_{loc}(\Omega)$, we have $\Delta_r^wu\to L_wu$ in $L^p_{loc}(\Omega)$. More specifically, we have the following theorem.

\begin{theorem}\label{thm:weight2}
	Let $p\in (1,\infty)$. Then $W^{1,p}_{loc}(\Omega)\cap \AMV^p_{loc}(\Omega_w)=W^{2,p}_{loc}(\Omega)$ and, for every $u\in W^{2,p}_{loc}(\Omega)$, we have that $\Delta_r^wu\stackrel{r\to 0}{\longrightarrow} L_wu$ in $L^p_{loc}(\Omega)$.
\end{theorem}

We will prove Theorem \ref{thm:weight2} by reducing it to the unweighted case (using Lemma \ref{lem:rep}), where the result follows by standard theory in a straightforward way.

The following observation is well known, but we briefly recall the proof.

\begin{lemma}\label{lem:c2uni}
	If $u\in C^2(\Omega)$, then
	\[
	\Delta_ru\to \frac 12 \div(M\nabla u)
	\]
	locally uniformly in $\Omega$, as $r\to 0$.
\end{lemma}
\begin{proof}
	Let $x\in \Omega$ and $r<\dist(x,\partial \Omega)$. The Taylor expansion of $u$ yields
	\[
	u(x+rz)-u(x)=r\sum_{i=1}^n\partial_iu(x)z_i+\frac 12 r^2\sum_{i,j=1}^n\partial_{ij}u(x)z_iz_j+E_r(x,z),
	\]
	where $E_r(x,z)/r^2\to 0$ locally uniformly in $\Omega\times \bar B^n$, as $r\to 0$. Thus we have 
	\begin{align*}
		\Delta_ru(x)=\vint_{B^n}\frac{u(x+rz)-u(x)}{r^2}\ud z=\frac 12 \sum_{i,j=1}^n\partial_{ij}u(x)\vint_{B^n}z_iz_j\ud z+\vint_{B^n}\frac{E_r(x,z)}{r^2}\ud z,
	\end{align*}
	from which the claim follows.
\end{proof}
\begin{prop}\label{prop:unweight}
	Let $p>1$. Then $\AMV_{loc}^p(\Omega)=W^{2,p}_{loc}(\Omega)$ and, for each $u\in W^{2,p}_{loc}(\Omega)$, we have that 
	\begin{equation}\label{eq-prop:unw}
		\Delta_ru\to \frac 12 \div(M\nabla u)
	\end{equation}
	in $L^p_{loc}(\Omega)$, as $r\to 0$.
\end{prop}
\begin{proof}
	Assume $u\in W^{2,p}_{loc}(\Omega)$ and set
	\[
	R(x,y):=u(y)-u(x)-\sum_{i=1}^n\partial_iu(x)(y-x)_i,\quad \hbox{for $\mu$-a.e }\,x,y\in \Omega.
	\]
	Let $\Omega'\subset \Omega$ be a compactly contained $W^{2,p}$-extension domain. Then, for a ball (respectively, a unit sphere) $B_r^E$ ($\partial B_1^E$) in the Euclidean norm, by \cite[Thm 2.5, Thm 4.1 and (21)]{biki} we have that
	\begin{align}\label{eq:limex}
		\lim_{r\to 0}\int_{\Omega'}\vint_{B_r^E(x)}\left|\frac{R(x,y)}{r^2}\right|^p\ud y\ud x= c(n,p) \int_{\Omega'}\int_{\partial B_1^E(0)}\left|2\sum_{i\le j}\partial_{ij}u(x)e_ie_j \right|^p\ud e\ \ud x.
	\end{align}
	By Remark \ref{rmk:sym} we see that 
	\begin{equation}\label{eq:delta}
		\vint_{B_r(x)}\frac{R(x,y)}{r^2}\ud y=\Delta_ru(x),\quad \bar B_r(x)\subset \Omega.
	\end{equation}
	Together with \eqref{eq:limex} and the fact that $\|\cdot\|$ is comparable to the Euclidean norm, \eqref{eq:delta} implies that
	\[
	\limsup_{r\to 0}\int_{\Omega'}|\Delta_ru|^p\ud x\le \limsup_{r\to 0}\int_{\Omega'}\vint_{B_r(x)}\left|\frac{R(x,y)}{r^2}\right|^p\ud y\ud x\le c\lim_{r\to 0}\int_{\Omega'}\vint_{B_r^E(x)}\left|\frac{R(x,y)}{r^2}\right|^p\ud y\ud x<\infty.
	\]
	Since any point $x\in \Omega$ has a neighbourhood $\Omega'$ which is a $W^{2,p}$-extension domain, it follows that $u\in \AMV_{loc}^p(\Omega)$.
	
	Conversely, suppose $u\in \AMV_{loc}^p(\Omega)$. Then, for any positive sequence $(r_m)$ converging to zero there is a further subsequence and a function $g\in L^p_{loc}(\Omega)$ such that $\Delta_{r_m}u\rightharpoonup g$ weakly in $L^p_{loc}(\Omega)$ as $m\to \infty$. In particular, for any $\varphi\in C^2_c(\Omega)$, we have
	\[
	\int_\Omega \varphi g \ud x=\lim_{m\to\infty}\int_\Omega\varphi \Delta_{r_m}u\ud x=\lim_{m\to\infty}\int_\Omega\Delta_{r_m}\varphi u\ud x=\frac 12 \int_\Omega \div(M\nabla\varphi)u\ud x,
	\]
	since $\Delta_r\varphi\to \frac 12 \div(M\nabla\varphi)$ locally uniformly in $\Omega$, as $r\to 0$. This shows that $g$ is unique and agrees with the distribution $\frac 12 \div(M\nabla u)$ on $\Omega$. Thus $\div(M\nabla u)\in L^p_{loc}(\Omega)$ and we have that $u\in W^{2,p}_{loc}(\Omega)$, see e.g.~\cite[Theorem 6.29]{grubb}, applied for the differential operator $P=\Delta$.
	
	It remains to prove the convergence in~\eqref{eq-prop:unw}. For this assume $u\in W^{2,p}_{loc}(\Omega)$. Denote by $\nabla^2u$ the matrix of second weak partial derivatives of $u$ and by $Q_{s}(x)$ a concentric cube centered at $x$ with side length $s$. Then, by ~\cite[Thm 3.3. lemma 3.4]{bhs} we observe that there exist constants $C,\sigma\ge 1$, such that $C=C(n)$, for which 
	\begin{equation}\label{eq:ub}
		|\Delta_ru(x)|\le \frac{C(n)}{r^2}\int_{Q_{\sigma r}(x)}\frac{|\nabla^2u|(y)}{|x-y|^{n-2}}\ud y\le {\mathcal M}_{\sqrt{n}\sigma r}|\nabla^2u|(x), \quad B_{\sigma r}(x)\subset \Omega.
	\end{equation}
	Let $\Omega'\subset\Omega$ be compactly contained and let $(u_m)$ be a sequence of smooth functions converging to $u$ in $W^{2,p}(\Omega')$ as $m\to \infty$. In particular, $|\nabla^2(u-u_m)|\to 0$ and $\div(M\nabla(u-u_m))\to  0$ in $L^p(\Omega')$ as $m\to\infty$. By Lemma \ref{lem:c2uni} and \eqref{eq:ub} we obtain
	\begin{align*}
		&\limsup_{r\to \infty}\|\Delta_ru-\frac 12 \div(M\nabla u)\|_{L^p(\Omega')}\\
		\le& \limsup_{r\to 0}\left( \|\Delta_r(u-u_m)\|_{L^p(\Omega')}+\|\Delta_r u_m+\frac 12\div(M\nabla u_m)\|_{L^p(\Omega')}+\frac 12 \|\div(M\nabla(u-u_m))\|_{L^p(\Omega')}\right)\\
		\le & C\limsup_{r\to 0}\|{\mathcal M}_{\sqrt{n}\sigma r}|\nabla^2(u-u_m)|\|_{L^p(\Omega')}+\frac 12 \|\div(M\nabla(u-u_m))\|_{L^p(\Omega')}\\
		\le & C'\|\nabla^2(u-u_m)\|_{L^p(\Omega')}+\frac 12\|\div(M\nabla(u-u_m))\|_{L^p(\Omega')}.
	\end{align*}
	Upon letting $m\to \infty$ the claim follows.
\end{proof}

In order to deal with the ``inner product'' term in \eqref{eq:wlapl} we have the following lemma.

\begin{lemma}\label{lem:inprod}
	If $f,g\in W^{1,1}_{loc}(\Omega)$ and in addition $g\in L^n_{loc}(\Omega)$, then
	\begin{align*}
		\langle f,g\rangle_r\stackrel{r\to 0}{\longrightarrow} \langle M\nabla f,\nabla g\rangle
	\end{align*}
	pointwise almost everywhere in $\Omega$.
	
\end{lemma}
\begin{proof}
	The Sobolev embedding implies that $f\in L_{loc}^{n/(n-1)}(\Omega)$. Thus  $fg\in L^1_{loc}(\Omega)$ and $\langle f,g\rangle_r$ is finite. Sobolev functions in $W^{1,1}_{loc}(\Omega)$ satisfy the following approximation by tangent planes for almost every $x\in \Omega$, cf. \cite[Subsection 6.1.2, Theorem 2]{gar-ev}:
	\begin{align}\label{eq:ptest}
		\lim_{r\to 0}\left(\vint_{B_1^n(x)}\left|\frac{f(x+rz)-f(x)-\nabla f(x)\cdot(rz)}{r}\right|^{n/(n-1)}\ud z\right)^{(n-1)/n}=0.
	\end{align}
	It follows that
	\begin{equation}\label{eq:fineprop}
		\lim_{r\to 0}\left(\vint_{B_1^n(x)}\left|\frac{f(x+rz)-A_rf(x)-\nabla f(x)\cdot(rz)}{r}\right|^{n/(n-1)}\ud z\right)^{(n-1)/n}=0,
	\end{equation}
	since
	\begin{align*}
		&\left(\vint_{B_1^n(x)}\left|\frac{f(x+rz)-A_rf(x)-\nabla f(x)\cdot(rz)}{r}\right|^{n/(n-1)}\ud z\right)^{(n-1)/n}\\
		\le& \left(\vint_{B_1^n(x)}\left|\frac{f(x+rz)-f(x)-\nabla f(x)\cdot(rz)}{r}\right|^{n/(n-1)}\ud z\right)^{(n-1)/n}+\left|\frac{A_rf(x)-f(x)}{r}\right|\\
		\le & 2\left(\vint_{B_1^n(x)}\left|\frac{f(x+rz)-f(x)-\nabla f(x)\cdot(rz)}{r}\right|^{n/(n-1)}\ud z\right)^{(n-1)/n}.
	\end{align*}

	For the next calculations we use the shorthand 
	\[
	R_r^f(x,z):= \frac{f(x+rz)-A_rf(x)-\nabla f(x)\cdot(rz)}{r},\quad \bar B_r(x)\subset \Omega.
	\]
	The definition of the matrix $M$ yields
	\begin{align*}
		\langle M\nabla f(x),\nabla g(x)\rangle=\vint_{B_1^n(x)}\sum_{i,j=1}^n\partial_if(x)\partial g_j(x)z_iz_j\ud z=\vint_{B_1^n(x)}(\nabla f(x)\cdot z)(\nabla g\cdot z)\ud z
	\end{align*}
	which may be expanded to 
	\begin{align*}
		\langle M\nabla f(x),\nabla g(x)\rangle=&\vint_{B_1^n(x)}\frac{(f(x+rz)-A_rf(x))}{r}(\nabla g(x)\cdot z)\ud z-\vint_{B_1^n(x)}R_r^f(x,z)(\nabla g(x)\cdot z)\ud z\\
		=&\vint_{B_1^n(x)}\frac{f(x+rz)-A_rf(x)}{r}\frac{g(x+rz)-A_rg(x)}{r}\ud z\\
		&-\vint_{B_1^n(x)}[R_r^g(x,z)(\nabla f(x)\cdot z)+R_r^f(x,z)(\nabla g(x)\cdot z)]\ud z.
	\end{align*}
	Thus
	\begin{align*}
		\langle f,g\rangle_r(x)-\langle M\nabla f(x),\nabla g(x)\rangle=\vint_{B_1^n(x)}[R_r^g(x,z)(\nabla f(x)\cdot z)+R_r^f(x,z)(\nabla g(x)\cdot z)]\ud z
	\end{align*}
	tends to zero as $r\to 0$ for almost every $x\in \Omega$, by \eqref{eq:fineprop}.
\end{proof}

The proof of Theorem \ref{thm:weight2} can now be reduced to the unweighted case using the results above.

\begin{proof}[Proof of Theorem \ref{thm:weight2}]
	We first observe that, if $u\in W^{1,p}_{loc}(\Omega)$, then
	\begin{align*}
		|\langle u,w\rangle_r(x)|\le & \vint_{B_r(x)} \frac{|u-A_ru(x)|}{r}\frac{|w-A_rw(x)|}{r}\ud y \\
		\le &\|\nabla w\|_{L^\infty(B_r(x))}\vint_{B_r(x)}|u-A_ru(x)|/r\ud y\le C\|\nabla w\|_{L^\infty(B_r(x))}\M_r|\nabla u|(x),
	\end{align*}
	where by $\M_r$ we denote the $r$-restricted maximal function. It follows that, for each compactly contained $\Omega'\subset \Omega$ there exists a function $g\in L^p(\Omega')$ such that 
	\begin{equation}\label{eq:ptb}
		|\langle u,w\rangle_r|\le g
	\end{equation}
	almost everywhere in $\Omega'$. Together with Lemma \ref{lem:inprod}, \eqref{eq:ptb} implies that 
	\begin{equation}\label{eq:lpconv}
		\langle u,w\rangle_r\to \langle M\nabla u,\nabla w\rangle\quad \textrm{ in }\ L^p_{loc}(\Omega)
	\end{equation}
	
	Suppose $u\in W^{2,p}_{loc}(\Omega)$. Then $\langle u,w\rangle_r\to \langle M\nabla u,\nabla w\rangle$  and $\Delta_ru\to \frac 12 \div(M\nabla u)$ in $L^p_{loc}(\Omega)$. 
	Thus Lemma \ref{lem:rep} \eqref{eq:wlapl} implies that 
	\[
	\Delta_r^wu\stackrel{r\to 0}{\longrightarrow} L_wu
	\]
	in $L^p_{loc}(\Omega)$, which also implies that $u\in \AMV^p_{loc}(\Omega_w)$.
	
	Conversely, if $u\in W^{1,p}_{loc}(\Omega)\cap \AMV_{loc}^p(\Omega_w)$, Lemma \ref{lem:rep}\eqref{eq:wlapl} and \eqref{eq:ptb} imply that, for every compactly contained $\Omega'\subset \Omega$ there exists a function $g\in L^p(\Omega')$ such that 
	\[
	|\Delta_ru|\le |\Delta_r^wu|+\frac{1}{A_rw}g\quad\textrm{ on }\ \Omega'. 
	\]
	Thus $u\in \AMV^p_{loc}(\Omega)$ which, by Proposition \ref{prop:unweight}, implies that $u\in W^{2,p}_{loc}(\Omega)$.
\end{proof}
Using Lemma \ref{lem:rep} \eqref{eq:2ident} it is possible to prove that, for $w\in W^{2,\infty}_{loc}(\Omega)$, the classes $\AMV_{loc}^p(\Omega_w)$ agree with $W^{2,p}_{loc}(\Omega)$.
\begin{theorem}
	Suppose $w\in W^{2,\infty}_{loc}(\Omega_w)$, and $p\in (1,\infty)$. Then $\AMV^{p}_{loc}(\Omega_w)=W^{2,p}_{loc}(\Omega)$. For every $u\in \AMV_{loc}^p(\Omega_w)$ we have that $\Delta_r^wu\to L_wu$ in $L^p_{loc}(\Omega)$, as $r\to 0$.
\end{theorem}

\begin{proof}
	We sketch the proof and omit the details.
	
	Identity \eqref{eq:2ident} yields $|\Delta_r^w(uw)|\le A_rw|\Delta_r^wu|+|u||\Delta_rw|$ for $u\in \AMV_{loc}^p(\Omega_w)$. By Proposition \ref{prop:unweight} we have that $uw\in W^{2,p}_{loc}(\Omega)$ which implies $u\in W^{2,p}_{loc}(\Omega)$ since $w^{-1}\in W^{2,\infty}_{loc}(\Omega)$. Using \eqref{eq:2ident} and Proposition \ref{prop:unweight} we have the convergence 
	\begin{align*}
		\Delta_r^wu &\stackrel{r\to 0}{\longrightarrow} \frac {1}{2w}(\div(M\nabla(uw)))-u\div(M\nabla w))\\
		&=\frac{1}{2w}(u\div(M\nabla w)+2\langle \nabla w,M\nabla u\rangle+w\div(M\nabla u)-u\div(M\nabla w)\nonumber)=L_wu
	\end{align*}
	in $L^p_{loc}(\Omega)$, whenever $u\in W^{2,p}_{loc}(\Omega)$.
\end{proof}

\subsection{PDE characterization of amv-harmonicity}
We are now in a position to prove Theorem \ref{thm:weight}. For clarity, we formulate the following weak counterpart of Proposition \ref{prop:unweight} as a separate lemma. We omit the proof, which follows e.g. by an approximation argument from Proposition \ref{prop:unweight}, or via a direct argument in the spirit of Proposition \ref{prop:unweight}.

\begin{lemma}\label{lem:weak}
	Let $u\in W^{1,2}_{loc}(\Omega)$. Then $\Delta_ru\to \frac 12 \div(M\nabla u)$ weakly, i.e.
	\begin{align*}
		\lim_{r\to 0}\int_\Omega \varphi\Delta_ru\ud x=-\frac 12 \int\langle\nabla\varphi,M\nabla u\rangle\ud x
	\end{align*}
	for every $\varphi\in C_c^\infty(\Omega)$.
\end{lemma}
\begin{proof}[Proof of Theorem \ref{thm:weight}]
	We first prove the equivalences (1)--(3). The implication (3)$ \implies$ (2) is clear, and (2) $\implies$ (1) follows from Lemmas \ref{lem:weak} and \ref{lem:rep}\eqref{eq:wlapl}. Indeed, together with \eqref{eq:lpconv} in the proof of Theorem \ref{thm:weight2}, we have that
	\begin{align*}
		0=&\lim_{r\to 0}\int_\Omega\varphi \Delta_r^wu\ud x=\lim_{r\to 0}\left(\int_\Omega \varphi \Delta_ru\ \ud x+\int_\Omega\frac{\varphi}{A_rw}\langle u,w\rangle_r\ \ud x \right)\\
		=&-\frac 12\int_\Omega \langle\nabla \varphi,M\nabla u\rangle\ud x+\int_\Omega \frac{\varphi}{w}\langle \nabla w,M\nabla u\rangle\ud x
	\end{align*}
	for every $\varphi\in C_c^\infty(\Omega)$, which is the weak formulation of $L_wu=0$. 
	
	Thus it suffices to prove the implication (1) $\implies$ (3). By elementary regularity theory, see e.g. ~\cite[Thm 8.8]{gt}, a weak solution $u\in W^{1,2}_{loc}(\Omega)$ to $L_wu=0$ satisfies $u\in W^{2,2}_{loc}(\Omega)$. Theorem \ref{thm:weight2} now directly implies that (3) holds.
	
	\bigskip\noindent It remains to prove the equivalence of (1)--(3) with (4) assuming smoothness of $w$. Since (4) $\implies$ (2), it suffices to prove that (1) $\implies$ (4). Since weak solutions of $L_wu=0$ with $w\in C^\infty(\Omega)$ are smooth (see e.g. \cite[Cor 8.11]{gt}), Lemma \ref{lem:c2uni} and \ref{lem:rep} \eqref{eq:2ident} imply that
	\[
	\Delta_r^wu=\frac{1}{A_rw}(\Delta_r(uw)-u\Delta_rw)\to \frac{1}{2w}(\div(M\nabla(uw))-u\div(M\nabla w))=L_wu=0
	\]
	locally uniformly in $\Omega$, as $r\to 0$.
\end{proof}

\bibliographystyle{plain}

\end{document}